\documentclass{amsart}
\usepackage{amsmath,amsfonts,latexsym,amssymb,amsthm, verbatim}
\usepackage{booktabs,setspace}
\usepackage{placeins}
\usepackage{rotating}
\usepackage{url}
\usepackage{tikz-cd}
\usetikzlibrary{matrix,arrows,decorations.pathmorphing}

\newtheorem{tm}{Theorem}
\newtheorem{proposition}[tm]{Proposition}
\newtheorem{lemma}[tm]{Lemma}

\theoremstyle{definition}

\theoremstyle{remark}
\newtheorem*{acknowledgments}{Acknowledgments}
\newtheorem{remark}[tm]{Remark}

\newcommand{\R}{\mathbb{R}}
\newcommand{\Q}{\mathbb{Q}}
\newcommand{\Qbar}{\overline{\Q}}
\newcommand{\C}{\mathbb{C}}
\newcommand{\Z}{\mathbb{Z}}

\newcommand{\F}{\mathbb{F}}
\newcommand{\PP}{\mathbb{P}}
\newcommand{\HH}{\mathbb{H}}

\newcommand{\cC}{\mathcal{C}}
\newcommand{\cL}{\mathcal{L}}

\newcommand{\an}{\mathrm{an}}
\newcommand{\dd}{\mathrm{d}}

\newcommand{\EisE}{\mathop{\mathbb{E}\null}\nolimits}

\DeclareMathOperator{\Gal}{Gal}
\DeclareMathOperator{\GL}{GL}

\DeclareMathOperator{\SL}{SL}

\DeclareMathOperator{\Norm}{Norm}
\DeclareMathOperator{\Res}{Res}
\DeclareMathOperator{\Spec}{Spec}

\DeclareMathOperator{\rk}{rk}

\DeclareMathOperator{\Si}{S}
\DeclareMathOperator{\SQ}{SQ}

\def\isom{\buildrel\sim\over\longrightarrow}
\def\lsup#1#2{{}^{#1}\!{#2}}
\def\smallmat#1#2#3#4{\bigl(\begin{smallmatrix} #1 & #2 \\
  #3 & #4 \end{smallmatrix}\bigr)}

\begin{document}

\title[Quadratic points on modular curves]{Hyperelliptic modular curves $X_0(n)$ and isogenies of elliptic curves over quadratic fields}
\author{Peter Bruin}
\address{Mathematics Institute\\ Zeeman Building\\ University of Warwick\\ Coventry CV4 7AL\\ United Kingdom}
\curraddr{Universiteit Leiden\\ Mathematisch Instituut\\ Postbus 9512\\ 2300 RA \ Leiden\\ Netherlands}
\email{P.J.Bruin@math.leidenuniv.nl}
\author{Filip Najman}
\address{Department of Mathematics\\ University of Zagreb\\ Bijeni\v cka cesta 30\\ 10000 Zagreb\\ Croatia}
\email{fnajman@math.hr}

\begin{abstract}
We study elliptic curves over quadratic fields with isogenies of certain degrees. Let $n$ be a positive integer such that the modular curve $X_0(n)$ is hyperelliptic of genus $\ge2$ and such that its Jacobian has rank~$0$ over~$\Q$.
We determine all points of $X_0(n)$ defined over quadratic fields, and we give a moduli interpretation of these points.
We show that, with finitely many exceptions up to $\Qbar$-isomorphism, every elliptic curve over a quadratic field~$K$ admitting an $n$-isogeny is $d$-isogenous, for some $d\mid n$, to the twist of its Galois conjugate by a quadratic extension $L$ of~$K$; we determine $d$ and~$L$ explicitly, and we list all exceptions.
As a consequence, again with finitely many exceptions up to $\Qbar$-isomorphism, all elliptic curves with $n$-isogenies over quadratic fields are in fact $\Q$-curves.
\end{abstract}

\maketitle

\section{Introduction}

The study of possible torsion groups of elliptic curves over number fields has seen a lot of progress in the last few decades. See, just to mention some, \cite{maz1} for results over $\Q$, \cite{Kam1,km} for results over quadratic fields, \cite{jks, naj} for results over cubic fields, and \cite{jkp} for results over quartic fields. We also mention the uniform boundness conjecture, proved by Merel \cite{mer}, which states that the order of the torsion group is bounded from above by an integer depending only on the degree of the number field.

Unfortunately, much less is known about the possible degrees of isogenies of elliptic curves over number fields.  A complete classification of possible isogeny degrees is only known over $\Q$. Mazur \cite{maz2} found all the isogenies of prime degree and gave a list of isogeny degrees which he believed to be complete, after which Kenku proved in a series of papers \cite{ken1,ken2,ken3,ken4} that the list is indeed complete.

The goal of this paper is to classify elliptic curves over quadratic fields with isogenies of certain degrees~$n$ and to investigate their properties.  We will try to obtain as much information as possible on these curves.  In particular, for many values of~$n$, we classify all such curves and explicitly describe the isogenies.

We say that an elliptic curve $E$ over a field~$K$ has an \emph{$n$-isogeny\/} if it has an isogeny with cyclic kernel of order~$n$ defined over~$K$.  As usual, the (compact) modular curve over~$\Q$ classifying elliptic curves with an $n$-isogeny will be denoted by $X_0(n)$, and the Jacobian of $X_0(n)$ will be denoted by $J_0(n)$.

To classify elliptic curves over quadratic fields admitting an $n$-isogeny, we determine all quadratic points on the curves $X_0(n)$ for $n$ such that $X_0(n)$ is hyperelliptic of genus at least 2 and such that the group of $\Q$-rational points of $J_0(n)$ is finite.  These assumptions are satisfied if and only if
$$
n \in \{ 22, 23, 26, 28, 29, 30, 31, 33, 35, 39, 40, 41, 46, 47, 48, 50, 59, 71 \}.
$$
Throughout this paper, unless stated otherwise, $n$ will denote an integer in the above set.

\begin{remark}
The curve $X_0(37)$ is also hyperelliptic, but $J_0(37)(\Q )$ has positive rank.
\end{remark}

After obtaining all the quadratic points on the aforementioned modular curves, we study their moduli interpretation and derive interesting consequences from this.  Quadratic points on $X_0(n)$ (except for the cusps) correspond to $\Qbar$-isomorphism classes of elliptic curves over quadratic fields with an $n$-isogeny.  Hence studying these points gives us information about the properties of such elliptic curves.

Taking inverse images of $\Q$-points under the hyperelliptic map $X_0(n)\to\PP^1(\Q)$ gives an infinite set of points of $X_0(n)$ over quadratic fields.  Apart from these, there are only finitely many quadratic points; we call these quadratic points \emph{exceptional}.

We show that for all elliptic curves over quadratic fields $K$ with a $28$- or $40$-isogeny, the quadratic field $K$ is real, except in a few explicitly listed cases.

Recall that a \emph{$\Q$-curve} over a number field $K\subset\Qbar$ is an elliptic curve $E$ over~$K$ that is $\Qbar$-isogenous to all of its Galois conjugates $\lsup\sigma E$ for $\sigma\in\Gal(\Qbar/\Q)$.  If $K$ is quadratic and $\sigma$ is the non-trivial automorphism of~$K$, then $E$ is a $\Q$-curve if and only if $E$ is $\Qbar$-isogenous to $\lsup\sigma E$.  We prove that all elliptic curves over quadratic fields obtained from points of $\PP^1(\Q)$ as above are in fact $\Q$-curves.  The fact that there are only finitely many exceptional quadratic points means that, up to $\Qbar$-isomorphism, there are only finitely many elliptic curves over quadratic fields with an $n$-isogeny that are not $\Q$-curves.

Let $E$ be a $\Q$-curve without complex multiplication over a quadratic field~$K$ with an $n$-isogeny, arising from the inverse image of a point of $\PP^1(\Q)$ as above.  Then $E$ is not in general $K$-isogenous to its Galois conjugate~$\lsup\sigma E$, but it is by construction $K$-isogenous to a quadratic twist of $\lsup\sigma E$.  It follows that $E$ becomes isogenous to $\lsup\sigma E$ after base change to a quadratic extension $L$ of~$K$.  We explicitly construct such an extension~$L$.  Moreover, we prove that $E$ has even rank over~$L$.

Similar results were proved in \cite{bbdn} for elliptic curves $E$ over quadratic fields~$K$ with torsion group $\Z/n\Z$ for $n\in\{13,16,18\}$.  These follow from arithmetic properties of the relevant modular curves $X_1(n)$.  In \cite{bbdn} it was also proved that for $n=13$ and $n=18$, the endomorphism ring of the restriction of scalars $\Res_{K /\Q} E$ of such a curve~$E$ contains an order in a quadratic field, which implies that $\Res_{K/\Q} E(\Q)\simeq E(K)$ has even rank.  We say that $E$ has \emph{false complex multiplication}; see \cite{bbdn} for a precise definition.  In all the cases we consider, our construction implies that $E$ has false complex multiplication over~$L$ in the sense of~\cite{bbdn}.

\begin{remark}
The difficulty that $E$ is not $K$-isogenous to $\lsup\sigma E$, but only to a twist of~$\lsup\sigma E$, does not happen for the elliptic curves with prescribed torsion studied in~\cite{bbdn}.  This is explained by the fact that $X_0(n)$ is only a coarse moduli space, while $X_1(n)$ is a fine moduli space.
\end{remark}

We now give an overview of the structure of the paper.

In Section \ref{sec:qp}, we compute $J_0(n)(\Q)$ and describe all the quadratic points on $X_0(n)$. We apply these results to show that for $n=28$ and $40$ almost all elliptic curves with $n$-isogenies are defined over real quadratic fields (Theorem~\ref{tm:fd}). We also prove that the largest $k$ such that there exist infinitely many $K$-isogeny classes over quadratic fields $K$ containing $k$ curves is $k=16$ (Theorem~\ref{tm:number_isogenies}).

In Section \ref{sec:interpretation_hi}, we give a moduli interpretation of the action on~$X_0(n)$ of the normalizer of $\Gamma_0(n)$ in $\GL_2(\Q)^+$, and in particular of the hyperelliptic involution~$\iota$.  This is well known in the case where $\iota$ is an Atkin--Lehner involution, but we have chosen to present this material from a general perspective that also encompasses the non-Atkin--Lehner involutions in the cases $n=40$ and $48$.

In Sections \ref{sec:modular_forms} and \ref{sec:moduli-interpretation}, we show that the non-exceptional quadratic points on $X_0(n)$ give rise to $\Q$-curves (Theorem~\ref{thm:even-rank}).  We also show how to compute the fields of definition of the isogenies between the conjugates of such $\Q$-curves.  This allows us to describe over which fields these elliptic curves acquire false complex multiplication, and by which rings.

In Section \ref{sec:exceptions}, we list all our computational results and give a moduli interpretation of the exceptional points.

The computer calculations were done using Magma \cite{mag}.  In particular, we made use of plane models of $X_0(n)$ given by polynomials in $\Q[x,y]$, and the $q$-expansions of the modular functions $x$ and~$y$ for every $n$.  These data were obtained from M. Harrison's Small Modular Curve database, which is included in Magma.

\begin{acknowledgments}
We wish to thank Samir Siksek for enlightening discussions on this topic, and Chi-Yun Hsu, Ariyan Javanpeykar and Drew Sutherland for pointing out mistakes in earlier versions of the paper. We thank the referees for their comments, which greatly improved the exposition in the paper. Finally, we would like to thank Tarun Dalal, who found a number of mistakes in the article after it was published.
\end{acknowledgments}

\section{Quadratic points on $X_0(n)$}
\label{sec:qp}

In this section we describe the set of all quadratic points on $X_0(n)$. We write $X=X_0(n)$ and $J=J_0(n)$.  We write $\iota$ for the hyperelliptic involution of $X$, and we fix a hyperelliptic map $x\colon X\to\PP^1$ and a cusp $C\in X(\Q)$.

Using Magma's functionality for $2$-descent \cite{sto}, one proves in all the cases that we consider that $J$ has rank~$0$ over $\Q$. One could alternatively prove that $J(\Q)$ has rank~$0$ using $L$-functions \cite[Section 3.10]{ste}. Thus in this section $J(\Q)$ is finite.

\subsection{Finding the quadratic points}
\label{subsec:finding-qp}

We first observe that finding the quadratic points on $X$ amounts to finding the rational points on the symmetric square $X^{(2)}$ of $X$, which classifies effective divisors of degree~2 on~$X$. The map
$$
\begin{aligned}
\phi\colon X^{(2)} &\longrightarrow J\\
D &\longmapsto [D-C-\iota(C)]
\end{aligned}
$$
is an isomorphism onto its image, except above $0$, and $\phi^{-1}\{0\}$ is isomorphic to $\PP^1$.
% (A point $Q$ of~$\PP^1$ corresponds to the divisor $x^{-1}\{Q\}$).
We have
$$
X^{(2)}(\Q)=\phi^{-1}\{0\}(\Q) \cup \phi^{-1}(J(\Q)\setminus\{0\}).
$$
Note that $\phi^{-1}(J(\Q)\setminus\{0\})$ is finite (since $J(\Q)$ is finite by assumption), so the set of exceptional quadratic points on~$X$ is finite.  Moreover, this set is easy to compute from $J(\Q)$; we will do this explicitly in Section~\ref{sec:exceptions}.

Let $K$ be a quadratic field, let $\sigma$ be the generator of $\Gal(K/\Q)$, and let $P$ be a point in $X(K)\setminus X(\Q)$ with rational $x$-coordinate. Then $\iota$ acts on $P$ in the same way as $\sigma$. Let $E$ be an elliptic curve over~$K$ in the $\Qbar$-isomorphism class of elliptic curves (together with a subgroup) corresponding to~$P$. As we will see, the moduli interpretation of~$\iota$ (for example, $\iota$ is an Atkin--Lehner involution $w_d$ for most values of~$n$) implies that the curves $E$ and $\lsup\sigma E$ are $\Qbar$-isogenous. Thus $E$ is a $\Q$-curve.

This construction is similar to the one with hyperelliptic modular curves $X_1(n)$ \cite[Section 4]{bbdn}, but there is a fundamental difference.  Namely, the curves $X_1(n)$ are fine moduli spaces, while $X_0(n)$ are only coarse moduli spaces. Consequently, the $\Q$-curves obtained from $X_1(n)$ in~\cite{bbdn} are $K$-isogenous to their Galois conjugates, while our curves obtained from $X_0(n)$ are only $K$-isogenous to their Galois conjugates up to twist, meaning that $E$ is $K$-isogenous to a twist of $\lsup\sigma E$, or in other words that $E$ and $\lsup\sigma E$ are $\Qbar$-isogenous.  We refer to Elkies~\cite[Section~3]{elk} for a detailed exposition of the relevant properties of the curves $X_0(n)$.

By what we have seen, determining $X^{(2)}(\Q)$ amounts to determining the finite group $J(\Q)$, which we will do in the next proposition. Note that for all curves $X$ that we consider, $X(\Q)$ consists entirely of cusps.  The \emph{cuspidal subgroup} of $J(\Q)$, denoted by $C_J$, is the subgroup of $J(\Q)$ generated by the classes of the divisors $P_1 - P_2$, where $P_1$, $P_2$ are $\Q$-rational cusps of $X$.

\begin{proposition}
Let $n$ be an integer such that $X_0(n)$ is hyperelliptic of genus at least~$2$ and such that $J_0(n)(\Q)$ has rank~$0$. Then $J_0(n)(\Q)$ is equal to its cuspidal subgroup $C_J$.
\end{proposition}

\begin{proof}
One can obtain an upper bound on the size of $J(\Q)$ by using the fact that, for a prime~$p$ of good reduction, the prime-to-$p$ part of $J(\Q)$ injects into $J(\F_p)$. The function \textsf{TorsionBound} in Magma does exactly this. In all cases except $n\in\{30,33,39,46,48\}$, this immediately shows that $C_J=J(\Q)$.

For $n\in\{30,33,39,46,48\}$, the bound obtained in this way is unfortunately larger than the order of the cuspidal subgroup.  For $n=30$ and $n=48$, we only obtain that the index $(J(\Q):C_J)$ is a divisor of~$4$, while for $n\in\{33,39,46\}$ the index is $1$ or~$2$.

We deal with the cases $n=33$ and $n=39$ by studying the group structures of the reductions, not just their orders.

Let $n=33$.
%The upper bound on the order of $J(\Q)$ obtained above is $200$.
One computes the group structure of $C_J$ to be $(\Z/10 \Z)^2$.  We compute
$$
J(\F_5)\simeq \Z /10\Z \oplus \Z /20 \Z
$$
and
$$
J(\F_7)\simeq (\Z /2\Z)^2 \oplus (\Z /10 \Z)^2.
$$
Since $J(\Q)$ injects into both of these groups, it follows that $J(\Q)=C_J$.

Similarly, in the case $n=39$, one computes
$$
\begin{aligned}
C_J&\simeq \Z/2\Z \oplus \Z/28\Z,\\
J(\F_5)&\simeq \Z/4\Z \oplus \Z/28\Z,\\
J(\F_7)&\simeq (\Z/2\Z)^3 \oplus \Z/28\Z,
\end{aligned}
$$
from which it follows that $J(\Q)=C_J$.

% We will explicitly do the case $n=48$ in detail, while the is handed in a similar way. (the case $n=30$ is dealt with in detail in \cite{fs}).

Let $n=30$.  We compute
% the upper bound on the size of $J(\Q)$ is $768$.
$$
\begin{aligned}
C_J&\simeq \Z /2 \Z \oplus\Z /4 \Z \oplus \Z /24\Z,\\
J(\F_7)&\simeq (\Z/2 \Z)^2 \oplus \Z/4 \Z \oplus \Z/ 48 \Z,\\
J(\F_{23})&\simeq \Z/2 \Z \oplus \Z/12 \Z \oplus (\Z/ 24 \Z)^2.
\end{aligned}
$$
Therefore $J(\Q)$ is isomorphic to a subgroup of $(\Z/2\Z)^2\oplus\Z/4\Z\oplus \Z /24\Z$.  Now the splitting field of $J[2]$ is $\Q(\sqrt{-3},\sqrt{5})$, and we can extract a basis for $J[2]$ from the classes of the divisors $P-Q$ where $P$ and~$Q$ are Weierstrass points of~$X$.  By computing the Galois action with respect to this basis and taking invariants, one shows that $J(\Q)[2]\simeq (\Z/2\Z)^3$.  This implies that $J(\Q)=C_J$, as claimed.

Next, let $n=46$.  We compute $C_J$ as
$$
C_J\simeq \Z/11\Z \oplus \Z/22\Z.
$$
Using reduction modulo 3 and~5, one shows that $J(\Q)$ is isomorphic to a subgroup of $(\Z/22\Z)^2$.  The splitting field of $J[2]$ is $\Q(\sqrt{-23},\alpha)$, where $\alpha^3-\alpha-1=0$.  Computing the Galois action on $J[2]$ as above, we obtain $J(\Q)[2]\simeq\Z/2\Z$, from which we conclude that $J(\Q) = C_J$.

Finally, let $n=48$. This case is harder, and we use the following method, proposed to us by Samir Siksek.  We have
$$
\begin{aligned}
C_J&\simeq \Z/4\Z \oplus \Z/8\Z\oplus \Z/8\Z,\\
J(\F_5)&\simeq \Z/2\Z \oplus \Z/4\Z \oplus (\Z/8\Z)^2.
\end{aligned}
$$
By computing the Galois action on $J[2]$ as above, we obtain $J(\Q)[2]\simeq(\Z/2\Z)^3$, and hence $J(\Q)[2]\subseteq C_J$. Suppose that $J(\Q)\ne C_J$; then there exist $P\in C_J$ and $Q\in J(\Q)\setminus C_J$ such that $2Q=P$. This implies that $P+2C_J$ is in the kernel of the map
$$C_J/2C_J\rightarrow J(\Q)/2J(\Q),$$
and hence also in the kernel of
\begin{equation}C_J/2C_J\rightarrow J(\F_5)/2J(\F_5).\label{redmap}\end{equation}
An explicit calculation shows that the map \eqref{redmap} is injective. Thus $P$ is in $2C_J$, say $P=2R$ for some $R\in C_J$. It follows that $P=2R$ and $2(Q-R)=0$ in $J(\Q)$. Since $C_J$ contains the complete $2$-torsion of $J(\Q)$, it follows that $Q\in C_J$, which is a contradiction. We conclude that $J(\Q)=C_J$.
%For the cases $n= 39, 48$ we use a similar procedure using the primes $p=17$ and $5$, respectively; for the case $n=30$, one uses the reductions modulo $7$ and modulo $11$.%, as explained in \cite{fs}.
\end{proof}

\subsection{Fields of definition of elliptic curves with $28$- or $40$-isogenies}
\label{subsec:28-40}

We now prove that all but finitely many $\Qbar$-isomorphism classes of elliptic curves over quadratic fields with $28$- or $40$-isogenies are defined over real quadratic fields. This follows from the description of all quadratic points that we have just obtained.

\begin{tm}
\label{tm:fd}
With finitely many exceptions (up to $\Qbar$-isomorphism), all elliptic curves over quadratic fields with a $28$- or $40$-isogeny are defined over real quadratic fields.
\end{tm}
\begin{proof}
The modular curves $X_0(28)$ and $X_0(40)$ admit the following equations (see for example \cite[Section~4.3]{gon-hyp} or \cite[Sections 4.3--4.4]{gal}):
\begin{align*}
X_0(28)\colon y^2 &= f_{28}(x)=x^6+10x^4+25x^2+28,\\
X_0(40)\colon y^2 &= f_{40}(x)=x^8+8x^6-2x^4+8x^2+1.
\end{align*}
Let $P$ be a non-exceptional quadratic point on $X_0(n)$, where $n=28$ or $40$.  Then $P$ is of the form $(x,\sqrt{f_n(x)})$ with $x\in \Q$. It clear that $f_n(x)>0$, so $P$ is defined over a real quadratic field.
\end{proof}

\begin{remark}
The exceptions mentioned in Theorem \ref{tm:fd} do occur: there exist elliptic curves over \emph{imaginary} quadratic fields with a $28$- or $40$-isogeny.  There are finitely many $\Qbar$-isomorphism classes of such curves, corresponding to the points in Tables \ref{tab:28} and \ref{tab:40}.
\end{remark}

\subsection{Number of elliptic curves in an isogeny class}
\label{subsec:number}

We now prove a result about isogeny classes of elliptic curves over quadratic fields. Kenku proved in \cite{ken5} that any $\Q$-isogeny class of elliptic curves contains at most $8$ curves. In \cite{s5}, the authors find isogeny classes over $\Q(\sqrt 5)$ containing 10 curves. We show that the largest $k$ such that there are infinitely many isogeny classes over quadratic fields containing $k$ curves is $k=16$, coming from points on $X_0(48)$. When counting isogeny classes, we count up to $\Qbar$-isomorphism.

\begin{tm}
\label{tm:number_isogenies}
There are infinitely many isogeny classes of elliptic curves over quadratic fields containing $16$ curves. For all $k>16$, there are only finitely many isogeny classes of elliptic curves over quadratic fields containing $k$ curves.
\end{tm}

\begin{proof}
Since $X_0(48)$ is hyperelliptic, it has infinitely many quadratic points.  Since there are only finitely many $\Qbar$-isomorphism classes of elliptic curves over quadratic fields with complex multiplication, there are in fact infinitely many quadratic points on $X_0(48)$ corresponding to elliptic curves without complex multiplication.  Let $E_1$ be an elliptic curve over a quadratic field~$K$ corresponding to such a point.

Let $f_{1,5}\colon E_1\to E_5$ be a $16$-isogeny (over~$K$) to another curve $E_5$. Then $f_{1,5}$ factors as $f_{4,5}\circ f_{3,4}\circ f_{2,3} \circ f_{1,2}$, where $f_{i,i+1}$ is a $2$-isogeny from $E_i$ to $E_{i+1}$ and $E_2$, $E_3$ and $E_4$ are elliptic curves over $K$.  We note that $E_2$, $E_3$ and $E_4$ all have two distinct $2$-isogenies over~$K$.
% Each $2$-isogeny over~$K$ corresponds to a $K$-rational point of order $2$, and hence
It follows that $E_2(K)$, $E_3(K)$ and $E_4(K)$ all have full $2$-torsion.
% But as $E_i[2]\simeq \Z/2\Z\oplus \Z/2\Z$ has three subgroups of order~$2$, it follows
This implies that each of these curves has a third $2$-isogeny, say to $E_6$, $E_7$ and $E_8$, respectively. We obtain the following isogeny diagram:
\begin{center}
\begin{tikzpicture}[%
  back line/.style={densely dotted},
  cross line/.style={preaction={draw=white, -,line width=6pt}},
  node distance=2.5cm]
  \node (A) {$E_1$};
  \node [right of=A, node distance =2.1cm] (B) {$E_2$};
  \node [right of=B, node distance =2.1cm] (C) {$E_3$};
  \node [right of=C, node distance =2.1cm] (D) {$E_{4}$};
  \node (E) [right of=D, node distance=2.1cm] {$E_{5}$};
  \node [below of=B, node distance =1.1cm] (B1) {$E_{6}$};
  \node [below of=C, node distance =1.1cm] (C1) {$E_7$};
  \node [below of=D, node distance =1.1cm] (D1) {$E_8$};
  \draw[dashed] ;
  \draw[dashed] ;
  \draw[cross line] (A)--(B)--(C)--(D)--(E);
  \draw[cross line] (B)--(B1);
  \draw[cross line] (C)--(C1);
  \draw[cross line] (D)--(D1);
  \draw[dashed] ;
\end{tikzpicture}
\end{center}
Each of the $E_i$ is $3$-isogenous (over $K$) to an elliptic curve $E_i'$, and hence we get 16 elliptic curves in this isogeny class.

The second claim follows from results of Bars \cite[Theorem 4.3]{bar}.
\end{proof}

\begin{remark}
The elliptic curve $E_1$ is $12$-isogenous to a twist of $\lsup\sigma{E_1}$, so $\lsup\sigma{E_1}$ is a twist of either $E_3'$ or $E_6'$.
It follows from the moduli interpretation of $\beta_{48}$ (see \S\ref{subsec:moduli-interpretation}) that the kernel of this isogeny is not in the kernel of a $48$-isogeny, ruling out $E_3'$.
We conclude that $\lsup\sigma{E_1}$ is a twist of $E_6'$.
\end{remark}

\section{Moduli interpretation of the normalizer of $\Gamma_0(n)$}

\label{sec:interpretation_hi}

Let $n$ be a positive integer, and let $B(\Gamma_0(n))$ denote the normalizer of $\Gamma_0(n)$ in the group $\GL_2(\Q)^+$ of $2\times 2$-matrices over~$\Q$ with positive determinant.  In this section, we describe a canonical action of $B(\Gamma_0(n))$ on~$X_0(n)$, and we give a moduli interpretation of this action.  From this we then derive a moduli interpretation of the hyperelliptic involution of $X_0(n)$.  This lies at the basis of the constructions and results of the next two sections.

\subsection{The action of $B(\Gamma_0(n))$}

\label{subsec:aut-modular-curves}

Let $\cC_n$ be the category defined as follows.  The objects of~$\cC_n$ are triples $(E\to S,C,\omega)$ consisting of an elliptic curve $E$ over a $\Q$-scheme $S$, a cyclic subgroup scheme $C$ of order~$n$ of~$E$ and a nowhere-vanishing global relative differential $\omega$ on~$E$.  A morphism $(E'\to S',C',\omega') \longrightarrow (E\to S,C,\omega)$ in~$\cC_n$ is a Cartesian diagram of schemes
$$
\begin{tikzcd}
E' \arrow{r}{\phi} \dar & E \dar \\
S' \rar & S
\end{tikzcd}
$$
that is compatible with the group structure and satisfies $\phi^*C=C'$ and $\phi^*\omega=\omega'$.  We will usually omit $S$ from the notation.

We let $B(\Gamma_0(n))$ act by equivalences of categories on $\cC_n$ as follows.  Let $\gamma\in B(\Gamma_0(n))$.  First we consider the case where $\gamma$ is a scalar matrix $\smallmat a00a$ with $a\in\Q^\times$.  For such $\gamma$, we define
$$
\gamma(E,C,\omega)=(E,C,a^{-1}\omega).
$$
Next we may assume that $\gamma$ is of the form $\smallmat abcd$ with $a,b,c,d\in\Z$.  We put $\delta=\det\gamma=ad-bc$; this is a positive integer.  After \'etale localization on~$S$, we can choose an isomorphism
$$
\phi\colon(\Z/n\delta\Z)^2_S\isom E[n\delta].
$$
We may assume that the subgroup $\langle\phi(0,\delta)\rangle$ generated by~$\phi(0,\delta)$ equals $C$.  We put
$$
C_\gamma=n\langle\phi(a,b),\phi(c,d)\rangle\subset E.
$$
This is a subgroup of order~$\delta$ whose structure (elementary divisors) is given by the Smith normal form of~$\gamma$.  Furthermore, we put
$$
E'=E/C_\gamma
$$
and
$$
C'=(\langle\phi(c,d)\rangle+C_\gamma)/C_\gamma\subset E';
$$
one easily sees that $C'$ is cyclic of order~$n$.  There exists a unique differential $\omega'$ on~$E/C_\gamma$ whose pull-back to~$E$ equals $\omega$.  We define
$$
\gamma(E, C, \omega)=(E', C', \omega').
$$

Because $X_0(n)$ is the coarse moduli space of pairs $(E, C)$ as above and the above construction commutes with scaling $\omega$, any $\gamma\in B(\Gamma_0(n))$ induces an automorphism
$$
\iota_\gamma\colon X_0(n)\isom X_0(n).
$$
This automorphism only depends on the image of~$\gamma$ in $B(\Gamma_0(n))/\Gamma_0(n)$.  Thus we get an action of the latter group by automorphisms of the curve $X_0(n)$ over~$\Q$.

% (To simplify, we may assume moreover that the greatest common divisor of $a$, $b$, $c$, $d$ equals~1.  Then the Smith normal form of the matrix~$\gamma$ is $\smallmat \delta001$, i.e., the elementary divisors are $(\delta,1)$.  Under this assumption, the group $C_\gamma$ is cyclic of order~$\delta$.)

\subsection{The complex perspective}

We now study an analogous construction over the complex numbers, where we can express elliptic curves using lattices.

Let $\cL_\Q$ denote the set of all group homomorphisms $\psi\colon\Q^2\to\C$ such that the group
$$
L_\psi=\psi(\Z^2)=\Z\psi(1,0)+\Z\psi(0,1)
$$
is a ``positively-oriented'' lattice in~$\C$, in the sense that $\psi(1,0)$ and $\psi(0,1)$ are $\R$-linearly independent and $\psi(1,0)/\psi(0,1)$ has positive imaginary part.

We define a free left action of $\GL_2(\Q)^+$ on $\cL_\Q$ as follows: given $\gamma\in\GL_2(\Q)^+$ and $\psi\colon\Q^2\to\C$ in~$\cL_\Q$, we define $\gamma\psi\in\cL_\Q$ by
$$
(\gamma\psi)(v) = \frac{1}{\det\gamma}\psi(v\gamma),
$$
where $v\gamma$ denotes the usual right action of~$\GL_2(\Q)^+$ on~$\Q^2$ (``row vectors'').  More concretely, the bases $(\omega_1,\omega_2)$ of~$L_\psi$ and $(\omega_1',\omega_2')$ of~$L_{\gamma\psi}$ defined by
$$
\begin{aligned}
\omega_1=\psi(1,0),&\quad \omega_2=\psi(0,1),\\
\omega_1'=(\gamma\psi)(1,0),&\quad \omega_2'=(\gamma\psi)(0,1),
\end{aligned}
$$
satisfy the relation
\begin{equation}
\omega_1'=\frac{1}{\det\gamma}(a\omega_1+b\omega_2),\quad
\omega_2'=\frac{1}{\det\gamma}(c\omega_1+d\omega_2).
\label{eq:gamma-on-basis}
\end{equation}

\begin{lemma}
\label{lemma:actionL}
\begin{enumerate}
\item Let $\psi$ and~$\psi'$ be in $\cL_\Q$.  Then $L_\psi=L_{\psi'}$ if and only if the orbits $\SL_2(\Z)\psi$ and $\SL_2(\Z)\psi'$ are equal.
\item Let $\psi$ be in~$\cL_\Q$, and let $\gamma$ and~$\gamma'$ be in $\GL_2(\Q)^+$.  Then $L_{\gamma\psi}=L_{\gamma'\psi}$ if and only if the cosets $\SL_2(\Z)\gamma$ and $\SL_2(\Z)\gamma'$ are equal.
\end{enumerate}
\end{lemma}

\begin{proof}
The first claim is easy to verify; the second one follows from the first and the fact that $\GL_2(\Q)^+$ acts freely on~$\cL_\Q$.
\end{proof}

Now let $n$ be a positive integer.  We write $\cL_n$ for the set of pairs of lattices $(L,L')$ in~$\C$ such that $L'\supseteq L$ and $L'/L$ is cyclic of order~$n$.  Let $\gamma_n$ be the matrix $\bigl(\begin{smallmatrix} n & 0 \\ 0 & 1 \end{smallmatrix}\bigr)$.  We define a map
$$
\begin{aligned}
\pi_n\colon\cL_\Q&\longrightarrow\cL_n\\
\psi&\longmapsto(L_\psi,L_{\gamma_n\psi}).
\end{aligned}
$$
We note that $L_{\gamma_n\psi}$ is the lattice spanned by $\psi(1,0)$ and $\frac{1}{n}\psi(0,1)$.

\begin{lemma}
The map $\pi_n$ is a quotient map for the left action of the subgroup $\Gamma_0(n)\subset\GL_2(\Q)^+$ on~$\cL_\Q$.
\end{lemma}

\begin{proof}
It is straightforward to check that $\pi_n$ is surjective.  We have
$$
\pi_n(\gamma\psi)=(L_{\gamma\psi},L_{\gamma_n\gamma\psi}).
$$
Hence $\pi_n(\gamma\psi)=\pi_n(\psi)$ if and only if $L_{\gamma\psi}=L_\psi$ and $L_{\gamma_n\gamma\psi}=L_{\gamma_n\psi}$.  By Lemma~\ref{lemma:actionL}, this condition is equivalent to $\gamma\in\SL_2(\Z)$ and $\gamma_n\gamma\gamma_n^{-1}\in\SL_2(\Z)$, which is in turn equivalent to $\gamma\in\Gamma_0(n)$.
\end{proof}

The above result yields a natural action of $B(\Gamma_0(n))$ on~$\cL_n$.  Viewing elements of~$\cL_n$ as values of $\pi_n$, we can describe this action as
$$
\gamma(L_\psi,L_{\gamma_n\psi})=(L_{\gamma\psi},L_{\gamma_n\gamma\psi})
\quad\text{for all }\gamma\in B(\Gamma_0(n))
\text{ and }\psi\in\cL_\Q.
$$
Furthermore, there is a map from $\cL_\Q$ to the upper half-plane~$\HH$ sending $\psi$ to $\psi(1,0)/\psi(0,1)$.  This descends to a map from $\cL_n$ to the non-compact analytic modular curve $\Gamma_0(n)\backslash\HH$.  From these observations and the relation~\eqref{eq:gamma-on-basis}, we obtain a commutative diagram
$$
\begin{tikzcd}
\cL_\Q \rar \arrow{d}{\pi_n} & \HH \dar \\
\cL_n \rar & \Gamma_0(n)\backslash\HH
\end{tikzcd}
$$
in which the upper horizontal map is compatible with the action of~$\GL_2(\Q)^+$ and the lower horizontal map is compatible with the action of~$B(\Gamma_0(n))$.

\subsection{Compatibility}

We now show that the constructions in the two preceding sections are compatible.  This allows us to reinterpret the action of $B(\Gamma_0(n))$ on the analytic modular curve $\Gamma_0(n)\backslash\HH$ as an action on objects of $\cC_n$.

Let $\cC_n^\an$ be the set of isomorphism classes of (analytified) objects of~$\cC_n$ over the base $S=\Spec\C$.  The action of~$B(\Gamma_0(n))$ on~$\cC_n$ induces an action on~$\cC_n^\an$.  We define a map
$$
\begin{aligned}
F_n\colon\cL_n&\longrightarrow\cC_n^\an\\
(L,L')&\longmapsto(\C/L,L'/L,2\pi i\,\dd z).
\end{aligned}
$$
The following result shows that $F_n$ respects the actions of $B(\Gamma_0(n))$ that we have defined on $\cL_n$ and on~$\cC_n^\an$, respectively.

\begin{proposition}
For all $\gamma\in B(\Gamma_0(n))$ and all $(L,L')\in\cL_n$, we have
$$
F_n(\gamma(L,L'))=\gamma(F_n(L,L')).
$$
\end{proposition}

\begin{proof}
Let $(L,L')\in\cL$, and let $\psi\in\cL_\Q$ be such that $\pi_n(\psi)=(L,L')$, so that $L=L_\psi$ and $L'=L_{\gamma_n\psi}$.  In~$\cC_n^\an$, we then have
$$
F_n(L,L')=(\C/L_\psi, L_{\gamma_n\psi}/L_\psi, 2\pi i\,\dd z).
$$
% Defining $\omega_1$, $\omega_2$, $\omega_1'$ and $\omega_2'$ as before, we have
% $$
% \begin{aligned}
% L_{\gamma\psi}&=\Z\omega_1' + \Z\omega_2',\\
% L_{\gamma_n\gamma\psi}&=\Z\omega_1' + \Z\omega_2'/n.
% \end{aligned}
% $$

We first assume that $\gamma$ is a scalar matrix $\smallmat a00a$ with $a\in\Q^\times$.  Then \eqref{eq:gamma-on-basis} implies
$$
L_{\gamma\psi}=a^{-1}L_\psi=a^{-1}L
\quad\text{and}\quad
L_{\gamma_n\gamma\psi}=a^{-1}L_{\gamma_n\psi}=a^{-1}L'.
$$
This implies
$$
\begin{aligned}
F_n(\gamma(L,L'))&=F_n(a^{-1}L,a^{-1}L')\\
&=(\C/a^{-1}L,a^{-1}L/a^{-1}L',2\pi i\,\dd z).
\end{aligned}
$$
On the other hand, we have
$$
\gamma(F_n(L,L'))=(\C/L,L/L',a^{-1}\cdot2\pi i\,\dd z).
$$
These two objects are isomorphic via the map $z\mapsto az$.

Next we assume that $\gamma$ is of the form $\smallmat abcd$ with $a,b,c,d\in\Z$ and $\delta=ad-bc>0$.  Then the lattices $L=L_\psi$ and~$L_{\gamma\psi}$ satisfy $L_{\gamma\psi}\supseteq L_\psi$ and $L_{\gamma\psi}/L_\psi$ has order~$\delta$.  We define a group isomorphism
$$
\begin{aligned}
\phi\colon(\Z/n\delta\Z)^2&\isom(\C/L)[n\delta]
={1\over n\delta}L/L\\
(x\bmod n\delta,y\bmod n\delta) &\longmapsto \psi(x/n\delta,y/n\delta)\bmod L.
\end{aligned}
$$
A straightforward computation shows that both $\gamma(F_n(L,L'))$ and $F_n(\gamma(L,L'))$ are equal to $(\C/L_{\gamma\psi}, L_{\gamma_n\gamma\psi}/L_{\gamma\psi}, 2\pi i\,\dd z)$ in~$\cC_n^\an$.  This concludes the proof.
\end{proof}

\subsection{Moduli interpretation of the hyperelliptic involution}

\label{subsec:moduli-interpretation}

We now describe an element $\gamma\in B(\Gamma_0(n))$ such that the hyperelliptic involution~$\iota$ of~$X_0(n)$ is the automorphism induced by~$\gamma$ as in \S\ref{subsec:aut-modular-curves}, and we give a moduli interpretation of~$\gamma$.

Suppose first that $\iota$ equals the Atkin--Lehner involution $w_d$ for some $d\mid n$ with $\gcd(d,n/d)=1$.  This is the case for all $n$ we consider except $n=40$ and $n=48$.  For the matrix $\gamma$ we can take any
$\smallmat{ad}{b}{n}{d}$ with $a,b\in\Z$ and $ad-b(n/d)=1$.

In \cite[Section 4]{ogg} it is proved that for $n=40$, the involution~$\iota$ is induced by the action of the matrix $\beta_{40}=\smallmat{-10}{1}{-120}{10}\in B(\Gamma_0(40))$ of determinant~$20$, and that for $n=48$ it is induced by the matrix $\beta_{48}=\smallmat{-6}{1}{-48}{6}\in B(\Gamma_0(48))$ of determinant~$12$.

By the description of the action of $B(\Gamma_0(n))$ in \S\ref{subsec:aut-modular-curves}, the element $\gamma$ equips every elliptic curve with a cyclic subgroup $C$ of order~$n$ with a second cyclic subgroup $C_\gamma$ of order~$\delta$, where $\delta=\det\gamma$, and hence with a $\delta$-isogeny.  If $\iota=w_d$, then we have $C_\gamma=(n/d)C$; if $n=40$, then $C_\gamma$ is a subgroup of order~$20$ such that $C\cap C_\gamma$ has order~$10$; if $n=48$, then $C_\gamma$ is a subgroup of order~$12$ such that $C\cap C_\gamma$ has order~$6$.

\section{Modular forms}
\label{sec:modular_forms}

Let $n$ be a positive integer, and let $\gamma\in B(\Gamma_0(n))$.  In this section, we construct a modular form $A_\gamma$ of weight~$2$ for~$\Gamma_0(n)$, based on the moduli interpretation of~$\gamma$ described in \S\ref{subsec:aut-modular-curves}.  The reason for introducing $A_\gamma$ is that it allows us to explicitly compute the isogenies determined by the hyperelliptic involution of $X_0(n)$.  This generalizes the methods of Elkies~\cite{elk} to our setting.

More precisely, let $(E\to S, C, \omega)$ be an object of $\cC_n$, and let $C_\gamma$ be the cyclic subgroup of~$E$ determined by the moduli interpretation of~$\gamma$ as in \S\ref{subsec:aut-modular-curves}.  Let $E'$ be the quotient $E/C_\gamma$, equipped with the differential induced from~$E$.  In the situation of Section~\ref{sec:moduli-interpretation} below, we know explicit expressions for $E$, $E'$, $\delta$ and~$A_\gamma(E,\omega)$.  The subgroup~$C_\gamma$ can be computed from these data using an algorithm due to Elkies \cite[Section~3]{elk}.  We can then use V\'elu's formulae \cite{velu} to compute rational functions defining an isogeny $E\to E'$ with kernel~$C_\gamma$.

\subsection{Preliminaries}

\label{subsec:modform-prelim}

We begin by recalling the facts about modular forms that we need and introducing our notation.

A modular form of weight~$k$ for~$\Gamma_0(n)$ can be viewed (\`a la Katz) as a rule that to every object $(E\to S,C,\omega)$ of $\cC_n$, where we may restrict to the case where $S$ is an affine scheme $\Spec A$, assigns an element $F(E\to S,C,\omega)\in A$ that satisfies $F(E\to S,C,\lambda\omega)=\lambda^{-k} F(E\to S,C,\omega)$ for all $\lambda\in A^\times$, is compatible with morphisms in $\cC_n$ as in \S\ref{subsec:aut-modular-curves}, and is regular at the cusps in a suitable sense; see Katz \cite[Chapter~1]{katz} for details.

If $E$ is the Tate curve over $\C((q))$ with parameter~$\zeta$, equipped with the canonical subgroup $C=\mu_n$ and the canonical differential $\omega=\frac{\dd\zeta}{\zeta}$, then the $q$-expansion of~$F$ is $F(q) = F(E,C,\omega)\in\C((q))$.  The space of modular forms of weight~$2$ for~$\Gamma_0(n)$ is isomorphic to the space of logarithmic differentials on~$X_0(n)$; see for example Katz \cite[Appendix~1]{katz}.  If $F_\alpha$ is the modular form of weight~2 attached to a logarithmic differential~$\alpha$, the relation between $\alpha$ and~$F_\alpha$ can be expressed in terms of the $q$-expansion of~$F_\alpha$ as $\alpha=F_\alpha(q)\frac{\dd q}{q}$.

Over $\C$, modular forms $F$ as above correspond to functions~$f$ on the upper half-plane~$\HH$: if $(E,C,\omega) = (\C/(\Z\tau+\Z), \langle 1/n\rangle,2\pi i\,\dd z)$, then $F(E,C,\omega)=f(\tau)$.  The logarithmic differential corresponding to a modular form~$f$ is $f(\tau)\frac{\dd q}{q}$, where $q=\exp(2\pi i\tau)$.
% More precisely, writing the form as $f(\tau)(2\pi i\,\dd z)^2$, the corresponding differential is $f(\tau)\frac{\dd q}{q}=f(\tau)\cdot 2\pi i\,\dd\tau$.

The group $B(\Gamma_0(n))$ acts from the right on the space of modular forms of weight~$k$ for~$\Gamma_0(n)$ via
$$
(F|_k\gamma)(E,C,\omega)=F(\gamma(E,C,\omega)).
$$
If $\alpha$ is a logarithmic differential on~$X_0(n)$ and $\iota$ is the automorphism of~$X_0(n)$ induced by~$\gamma$, we have
\begin{equation}
F_\alpha|_2\gamma=(\det\gamma)F_{\iota^*\alpha}.
\label{eq:ks-compat}
\end{equation}

As usual, we write $\sigma_k(m)$ for the sum of the $k$-th powers of the positive divisors of~$m$.  Let $\EisE_4$ and $\EisE_6$ denote the standard Eisenstein series, normalized so that their $q$-expansions are
$$
\begin{aligned}
\EisE_4(q)&=\frac{1}{240}+\sum_{m=1}^\infty \sigma_3(m)q^m,\\
\EisE_6(q)&=-\frac{1}{504}+\sum_{m=1}^\infty \sigma_5(m)q^m.
\end{aligned}
$$
If $E$ is elliptic curve given by a Weierstrass equation in $x$ and~$y$, equipped with the standard differential $\omega=\frac{\dd x}{2y}$, these Eisenstein series are related to the usual coefficients $c_4$ and~$c_6$ by
$$
c_4=240\EisE_4(E,\omega)
\quad\text{and}\quad
c_6=504\EisE_6(E,\omega).
$$
We view $\EisE_4$ and~$\EisE_6$ as modular forms for~$\Gamma_0(n)$.  Furthermore, for every divisor $d$ of~$n$, let $\EisE_2^{(d)}$ be the modular form of weight~2 for $\Gamma_0(n)$ given by the $q$-expansion
$$
\EisE_2^{(d)}(q)=\EisE_2(q)-d\EisE_2(q^d),
$$
where
$$
\begin{aligned}
\EisE_2(q) &= -\frac{1}{24} + \sum_{m=1}^\infty \sigma_1(m)q^m \\
&= -\frac{1}{24} + \sum_{n=1}^\infty \frac{q^n}{(1-q^n)^2}.
\end{aligned}
$$

\subsection{The modular form $A_\gamma$}

Let $n$ be a positive integer, and let $\gamma\in B(\Gamma_0(n))$.  Let $(E\to S, C, \omega)$ be an object of $\cC_n$, and let $C_\gamma$ be the cyclic subgroup of~$E$ determined by the moduli interpretation of~$\gamma$ as in \S\ref{subsec:aut-modular-curves}.  We may assume that $S$ is an affine scheme $\Spec R$ and that $E$ is given by a short Weierstrass equation
$$
E\colon y^2=x^3+a_4x+a_6\qquad
\text{($a_4,a_6\in R$, $4a_4^3+27a_6^2\ne0$)}
$$
with $\omega=\dd x/(2y)$.  We put
$$
A_\gamma(E\to S,C,\omega)=\sum_{P\in C_\gamma\setminus\{0\}} x(P) \in R.
$$
One can check that this construction defines a modular form of weight~$2$ on $X_0(n)$.

% \subsection{Explicit expressions for $A_\gamma$}

Next, for all $n$ in the set mentioned in the introduction, we will derive an explicit expression for~$A_\gamma$ in terms of Eisenstein series, where $\gamma$ is the matrix inducing the hyperelliptic involution fixed in~\S\ref{subsec:moduli-interpretation}.

If $n$ is not 40 or~48, the hyperelliptic involution of~$X_0(n)$ equals the Atkin--Lehner involution $w_d$ for some $d\mid n$ with $\gcd(d,n/d)=1$.  One can show that
\begin{equation}
A_\gamma=-2d\EisE_2^{(d)};
\label{eq:A-Elkies}
\end{equation}
see Elkies \cite[Section~3]{elk}.

In the remaining cases $n=40$ and $n=48$, the hyperelliptic involution of $X_0(n)$ is not an Atkin--Lehner involution.  In these cases we have $\gamma=\beta_{40}$ and $\gamma=\beta_{48}$, respectively, as defined in~\S\ref{subsec:moduli-interpretation}.
\begin{proposition}
\label{prop:isoexplan}
The $q$-expansions of the modular forms $A_{\beta_{40}}$ and~$A_{\beta_{48}}$ are
$$
\begin{aligned}
A_{\beta_{40}}(q)
&=40(\EisE_2^{(5)}(q)-3\EisE_2^{(10)}(q)+\EisE_2^{(20)}(q)), \\
A_{\beta_{48}}(q)
&=24(\EisE_2^{(3)}(q)-3\EisE_2^{(6)}(q)+\EisE_2^{(12)}(q)).
\end{aligned}
$$
\end{proposition}

\begin{proof}
The computation is inspired by Elkies's proof of~\eqref{eq:A-Elkies}.  We compute the $q$-expansion by evaluating $A_{\beta_{40}}$ on the Tate curve $\C/q^\Z$ over~$\C((q))$ with coordinate~$\zeta$, equipped with the standard subgroup $\mu_n$ and the differential $\dd\zeta/\zeta$.  The subgroup~$C_{\beta_{40}}$ is cyclic of order~20, generated by $q^{1/2}\zeta_{20}$.  The $x$-coordinates occurring in the definition of~$A_{\beta_{40}}$ can then be viewed as values of the Weierstrass $\wp$-function.  From the definition of~$A_{\beta_{40}}$ and the classical formula
$$
\wp(z;\tau) = (2\pi i)^2 \biggl( -2\EisE_2(q) + \sum_{n=-\infty}^\infty\frac{q^n\zeta}{(1-q^n\zeta)^2} \biggr),
$$
where $q=\exp(2\pi i \tau)$ and $\zeta=\exp(2\pi i z)$, we deduce
$$
\begin{aligned}
A_{\beta_{40}}(q) &= \sum_{\zeta\in C_{\beta_{40}}\setminus\{0\}} \biggl( -2\EisE_2(q) + \sum_{n\in\Z} \frac{q^n\zeta}{(1-q^n\zeta)^2} \biggr) \\
&= -38\EisE_2(q) + \sum_{\zeta\in C_{\beta_{40}}} \sum_{n\in\Z}\frac{q^n\zeta}{(1-q^n\zeta)^2}.
\end{aligned}
$$
It is a straightforward but tedious formal computation to show that
$$
A_{\beta_{40}}(q) = -38\EisE_2(q) + \sum_{n\in\Z} \biggl( S_{10}(q^n) - S_{10}(q^{n+1/2}) + S_{20}(q^{n+1/2}) \biggr),
$$
where for $m\ge1$ the rational function $S_m\in\Q(\zeta_m)(x)$ is defined by
$$
S_m=\sum_{k=1}^{m-1} \frac{\zeta_m^k x}{(1-\zeta_m^k x)^2}.
$$
By expanding in a Taylor series around $x=0$, we obtain
$$
S_m=m^2\sum_{l\ge1}lx^{lm}-\sum_{l\ge1}lx^l\in\Q[[x]],
$$
and hence
$$
S_m=m^2\frac{x^m}{(1-x^m)^2} - \frac{x}{(1-x)^2}\in\Q(x).
$$
By substituting $x=\exp(t)$ and expanding in a Taylor series around $t=0$, we get
$$
S_m(1)=\frac{1-m^2}{12}.
$$
Furthermore, it is easy to check that $S_m$ is invariant under replacing $x$ by~$1/x$.  We therefore obtain
$$
\begin{aligned}
A_{\beta_{40}} &= -38\EisE_2(q) - \frac{99}{12} + 2\sum_{n\ge1} S_{10}(q^n) - 2\sum_{n\ge0} S_{10}(q^{n+1/2}) + 2\sum_{n\ge0} S_{20}(q^{n+1/2}) \\
&= -38\EisE_2(q) -\frac{99}{12}
 + 2\sum_{n\ge1} \biggl( 100 \frac{q^{10n}}{(1-q^{10n})^2} - \frac{q^n}{(1-q^n)^2} \biggr) \\
&\qquad - 2\sum_{n\ge0} \biggl( 100 \frac{q^{10n+5}}{(1-q^{10n+5})^2} - \frac{q^{n+1/2}}{(1-q^{n+1/2})^2} \biggr) \\
&\qquad + 2\sum_{n\ge0} \biggl( 400 \frac{q^{20n+10}}{(1-q^{20n+10})^2} - \frac{q^{n+1/2}}{(1-q^{n+1/2})^2} \biggr) \\
&= -40\EisE_2(q) + 200\EisE_2(q^{10}) - 200 \sum_{n\ge0} \frac{q^{10n+5}}{(1-q^{10n+5})^2} + 800\sum_{n\ge0} \frac{q^{20n+10}}{(1-q^{20n+10})^2} \\
&= -40\EisE_2(q) + 200\EisE_2(q^{10}) - 200(\EisE_2(q^5)-\EisE_2(q^{10})) + 800(\EisE_2(q^{10})-\EisE_2(q^{20})) \\
&= -40\EisE_2(q) - 200\EisE_2(q^5) + 1200\EisE_2(q^{10}) - 800\EisE_2(q^{20}).
\end{aligned}
$$
which gives the first claimed equality; the second follows from the definition of $\EisE_2^{(d)}$.

Similarly, the hyperelliptic involution on $X_0(48)$ is induced by the action of the matrix $\beta_{48}=\smallmat{-6}{1}{-48}{6}$ of determinant~$12$ on~$\cC_{40}$ and on~$\cL_{40}$.  A calculation analogous to the one above gives the claimed identity.
\end{proof}

\section{Moduli interpretation of quadratic points}

\label{sec:moduli-interpretation}

Let $n$ be such that $X_0(n)$ is hyperelliptic of genus $\ge2$ and such that $J_0(n)(\Q)$ has rank~$0$.  The purpose of this section is to give a moduli interpretation of the quadratic points of $X_0(n)$.
In particular, we show that if $E$ is an elliptic curve over a quadratic field~$K$ admitting an $n$-isogeny, then $E$ is $K$-isogenous to a quadratic twist of its Galois conjugate, with finitely many exceptions up to $\Qbar$-isomorphism.  The following theorem summarizes the results of this section.

\begin{tm}
\label{thm:even-rank}
Let $K$ be a quadratic field, and let $E$ be an elliptic curve over~$K$ without complex multiplication coming from a non-exceptional point of~$X_0(n)(K)$.  Then there is a quadratic extension $L$ of~$K$ (which can be explicitly computed) such that the following holds.
\begin{enumerate}
\item The curve $E$ is a $\Q$-curve that is completely defined over $L$.
\item The curve $E$ acquires even rank over $L$.
\end{enumerate}
\end{tm}

\subsection{The moduli interpretation}

Let $\iota$ be the hyperelliptic involution on~$X_0(n)$, and let $\gamma$ be the element of~$B(\Gamma_0(n))$ chosen in \S\ref{subsec:moduli-interpretation} so that $\gamma$ induces the hyperelliptic involution~$\iota$.  We put
$$
\delta=\det\gamma.
$$
We fix a non-zero meromorphic differential $\alpha$ on~$X_0(n)$ satisfying
$$
\iota^*\alpha=-\alpha.
$$
Let $f_\alpha$ be the cusp form of weight~2 corresponding to~$\alpha$ as in \S\ref{subsec:modform-prelim}.  Then we can uniquely express $\EisE_4$ and~$\EisE_6$ as
$$
\EisE_4=g_4^{} f_\alpha^2
\quad\text{and}\quad
\EisE_6=g_6^{} f_\alpha^3,
$$
where $g_4$, $g_6$ are meromorphic functions on $X_0(n)$.

\begin{lemma}
\label{lemma:gamma-eis}
The action of $\gamma$ on $\EisE_4$ and~$\EisE_6$ is given by
$$
{\EisE_4}|_4\gamma = (-\delta)^2 \frac{\iota^* g_4}{g_4} \EisE_4
\quad\text{and}\quad
{\EisE_6}|_6\gamma = (-\delta)^3 \frac{\iota^* g_6}{g_6} \EisE_6.
$$
\end{lemma}

\begin{proof}
By~\eqref{eq:ks-compat} we have $f_\alpha|_2\gamma = \delta f_{\iota^*\alpha}$; our choice of~$\alpha$ implies $f_{\iota^*\alpha} = f_{-\alpha} = -f_\alpha$.  The claim now follows from the identities ${\EisE_4}|_4\gamma=(\iota^* g_4)(f_\alpha|_2\gamma)^2$ and ${\EisE_6}|_6\gamma=(\iota^* g_6)(f_\alpha|_2\gamma)^3$.
\end{proof}

For the rest of this section, we fix the following data.  Let $K$ be a quadratic field, and let $\sigma$ be the non-trivial element of $\Gal(K/\Q)$.  Let $E$ be an elliptic curve over~$K$ together with a cyclic subgroup~$C$ of order~$n$.  Let $P$ be the $K$-rational point of~$X_0(n)$ determined by $(E,C)$.  We assume that $E$ does not have complex multiplication, and furthermore that the $K$-rational point of $X_0(n)$ defined by the pair~$(E,C)$ is not in the finite set of exceptional quadratic points on $X_0(n)$.

We fix a non-zero global differential $\omega$ on~$E$; this gives rise to a (non-uniquely determined) Weierstrass equation and to the usual $c$-coefficients $c_4$ and $c_6$, which only depend on~$\omega$.  We define $\mu$ and $\lambda$ in~$K^\times$ by
$$
\mu = \frac{21 g_6(P) c_4}{10 g_4(P) c_6}
\quad\text{and}\quad
\lambda = -\delta\frac{\sigma(\mu)}{\mu}.
$$
We define an extension $L$ of~$K$ (of degree 1 or~2) by
$$
L=K(\sqrt{\lambda}).
$$
We note that $L$ can also be written as $K\bigl(\sqrt{-\delta\Norm_{K/\Q}(\mu)}\bigr)$ and is therefore either $K$ itself or a $V_4$-extension of~$\Q$.

Let $E'$ be an elliptic curve over~$K$, equipped with a non-zero global differential~$\omega'$ (or equivalently given by a short Weierstrass equation), such that the $c$-coefficients of~$(E',\omega')$ are
$$
c_4' = \lambda^2 \sigma(c_4)
\quad\text{and}\quad
c_6' = \lambda^3 \sigma(c_6).
$$
We note that $E'$ is a quadratic twist of the Galois conjugate $\lsup\sigma E$ of~$E$, namely $E'=(\lsup\sigma E)^{(\lambda)}$.

\begin{proposition}
\label{prop:delta-isogeny}
In the above situation, there exists a $\delta$-isogeny $\mu\colon E\to E'$ with kernel~$C_\gamma$ satisfying $\mu^*\omega'=\omega$.
\end{proposition}

\begin{proof}
Let $\omega''$ be the differential induced by~$\omega$ on $E/C_\gamma$, and let $c_4''$ and $c_6''$ be the $c$-coefficients  corresponding to $(E/C_\gamma,\omega'')$.  We have to prove that $(E/C_\gamma,\omega'')$ is isomorphic to $(E',\omega')$.  It is enough to show that $(c_4',c_6')=(c_4'',c_6'')$, or equivalently that each of the two modular forms $\EisE_4$ and~$\EisE_6$ takes the same value on $(E/C_\gamma,\omega'')$ and $(E',\omega')$.

There is some cyclic subgroup $C'$ in~$E/C_\gamma$ such that
$$
\gamma(E,C,\omega)=(E/C_\gamma,C',\omega'').
$$
Using Lemma~\ref{lemma:gamma-eis} and the fact that $\EisE_4$ only depends on $(E,\omega)$ and $g_4$ only depends on $P$, we get
$$
\begin{aligned}
\EisE_4(E/C_\gamma,\omega'') &= \EisE_4(E/C_\gamma,C',\omega'') \\
&= ({\EisE_4}|_4\gamma)(E,C,\omega) \\
&= (-\delta)^2 \biggl(\frac{\iota^* g_4}{g_4}\biggr)(P) \EisE_4(E,C,\omega) \\
&= (-\delta)^2 \frac{g_4(\iota(P))}{g_4(P)} \EisE_4(E,\omega).
\end{aligned}
$$
This implies
$$
c_4''=(-\delta)^2 \frac{g_4(\iota(P))}{g_4(P)} c_4.
$$
Noting that $g_4(\iota(P))=\sigma(g_4(P))$, we rewrite this as
$$
c_4''=(-\delta)^2 \frac{\sigma(g_4(P)/c_4)}{g_4(P)/c_4} \sigma(c_4).
$$
Likewise,
$$
c_6''=(-\delta)^3 \frac{\sigma(g_6(P)/c_6)}{g_6(P)/c_6} \sigma(c_6).
$$
The fact that $E$ defines the point~$P$ on $X_0(n)$ implies that $E$ is a twist of the curve defined by $y^2=x^3-5g_4(P)x-\frac{7}{12}g_6(P)$, so there exists $\mu\in K^\times$ such that
\begin{equation}
240 g_4(P)=\mu^2 c_4
\quad\text{and}\quad
504 g_6(P)=\mu^3 c_6.
\label{eq:mu}
\end{equation}
This implies
$$
c_4'' = \biggl(-\delta\frac{\sigma(\mu)}{\mu}\biggr)^2\sigma(c_4)
\quad\text{and}\quad
c_6'' = \biggl(-\delta\frac{\sigma(\mu)}{\mu}\biggr)^3\sigma(c_6).
$$
Furthermore, it follows from~\eqref{eq:mu} that $\mu$ can be expressed as
$$
\mu=\frac{21 g_6(P) c_4}{10 g_4(P) c_6},
$$
as claimed.
\end{proof}

\begin{proof}[Proof of Theorem~\ref{thm:even-rank}]
We take $L$ to be the field defined above.  The first claim follows from Proposition~\ref{prop:delta-isogeny} and our assumption that $E$ does not have complex multiplication.  To prove the second claim, we note that the curves $E^{(\lambda)}$ and $\lsup\sigma E$ are $K$-isogenous, which implies
$$
\begin{aligned}
\rk E(L) &= \rk E(K)+\rk E^{(\lambda)}(K) \\
&= \rk E(K) + \rk \lsup\sigma E(K).
\end{aligned}
$$
It remains to observe that $E(K)$ and $\lsup\sigma E(K)$ are isomorphic.
\end{proof}

\begin{remark}
One can in fact show that $E$ acquires ``false complex multiplication'' over~$L$, in the sense of~\cite{bbdn}.  More precisely, let $M=\Q$ if $L=K$, and let $M$ be one of the quadratic subfields of~$L$ other than $K$ itself if $[L:K]=2$.  Then there is an action of the ring $\Z[\sqrt{m}]$ on the Weil restriction $A=\Res_{L/M} E$ and hence on the group $A(M)=E(L)$, where $m$ is either $\delta$ or $-\delta$.  In particular, since $\delta$ is not a square in any of the cases we consider in this paper, $E(L)$ has even rank.  The rank of~$E$ will also be even over many extensions of $L$; see \cite{bn} for details.
\end{remark}

% \begin{remark}
% For all modular curves considered in this article, we have $A_\gamma|_2\gamma=-\delta A_\gamma$ [TODO: check/prove this].  This means that we can take the differential form~$\alpha$ to be the one associated to~$A_\gamma$; namely, it satisfies $\iota^*\alpha=-\alpha$.
% \end{remark}

\begin{remark}
The question over which field the isogenies between a $\Q$-curve and its Galois conjugates are defined was studied from a somewhat different perspective by Gonz\'alez~\cite{gon}.  We leave it to the interested reader to compare the two approaches.
\end{remark}

\subsection{An example}

We now give an example to show how to explicitly compute the field of definition of an $n$-isogeny on an elliptic curve over a quadratic field corresponding to a non-exceptional quadratic point on~$X_0(n)$.  We take $n=22$ for simplicity, but the same method works in all cases.

The curve $X_0(22)$ has genus~2, and the Small Modular Curve database in Magma gives the equation
$$
X_0(22)\colon y^2 - x^3 y = -x^4 + 5 x^3 - 10 x^2 + 12 x - 8.
$$
The hyperelliptic involution~$\iota$ equals the Atkin--Lehner involution~$w_{11}$ and is induced by the matrix $\gamma=\smallmat{11}{5}{22}{11}$.  We fix the differential
$$
\alpha = \frac{\dd x}{(x - 2)(2y - x^3)}.
$$
Then we have
$$
\EisE_2^{(11)} = -\frac{5x^3+2x^2+12x+8}{12} f_\alpha.
$$
Furthermore,
$$
\EisE_4=g_4 f_\alpha^2
\quad\text{and}\quad
\EisE_6=g_6 f_\alpha^3,
$$
where $g_4$ and~$g_6$ are the rational functions defined by
$$
\begin{aligned}
240g_4 &= 120(-x^3 + 6x^2 + 4x + 8)y \\
&\qquad + 121x^6 - 484x^5 + 604x^4 - 352x^3 - 400x^2 + 2496x - 2240, \\
-504g_6 &= 36(-37x^6 + 236x^5 - 140x^4 + 1520x^3 + 368x^2 + 1408x - 448)y \\
&\qquad + 1331x^9 - 7986x^8 + 17304x^7 - 9832x^6 - 49632x^5 \\
&\qquad + 148704x^4 - 174720x^3 + 131712x^2 + 16128x - 179200.
\end{aligned}
$$

% We consider the points with $x=0$.  These are the points $P=(0,2\sqrt{-2})$ and $\iota(P) = \lsup\sigma P = (0,-2\sqrt{-2})$ over $K=\Q(\sqrt{-2})$; here $\sigma$ is the non-trivial element of $\Gal(K/\Q)$.  We have
% $$
% g_4(P)=8\sqrt{-2}-28/3
% \quad\text{and}\quad
% g_6(P)=64\sqrt{-2}+3200/9.
% $$
We consider the points with $x=-1$.  These are defined over the quadratic field~$K$ of discriminant~$-143$, and the points are $P=(-1,\beta)$ and $\iota(P)=\lsup\sigma P=(-1,-1-\beta)$, where $\beta^2+\beta+36=0$.  One of the elliptic curves in the family (consisting of quadratic twists) corresponding to the point $P$ is
$$
E\colon y^2 + xy + (1+\beta)y = x^3 - \frac{1}{2}x^2 + (74-28\beta)x + \frac{637-281\beta}{2}.
$$
(The class number of~$K$ is 10, and $E$ does not admit a global minimal model.)  The element $\mu\in K^\times$ happens to be $1$, so $E$ and $(\lsup\sigma E)^{(-11)}$, with $\sigma$ the non-trivial automorphism of~$K$, are related by an $11$-isogeny.  Furthermore, if $C\subset E$ is the canonical cyclic subgroup of order~$22$ and $\omega$ is the standard differential, the modular form~$A_\gamma$ takes the value $-77/6$ on $(E,C,\omega)$.  Using Elkies's algorithm, one computes that the kernel $C_\gamma$ of this $11$-isogeny is defined by the polynomial
$$
\begin{aligned}
P_{C_\gamma} &= x^5 + 6x^4 + (285+33\beta)x^3 - (1110+759\beta)x^2 \\
&\qquad + (40298-12496\beta)x - (13223+38324\beta).
\end{aligned}
$$
Using known algorithms, one can compute the rational functions defining the isogeny, but we will not write these down.

Finally, we note that both $E$ and its quadratic twist by $-11$ have rank~0; this can be shown by a 2-descent.  This implies that $E$ has rank~0 over~$L=K(\sqrt{-11})$, which is consistent with Theorem~\ref{thm:even-rank}.

% Let $U$ be the complement of the cusps and the vanishing locus of the differential~$\alpha$.  Then $U$ is symmetric under~$\iota$, and the functions $g_4$, $g_6$, $\iota^* g_4$ and $\iota^* g_6$ do not have poles on~$U$.  We consider the elliptic curves $E$ and~$E'$ over~$U$ defined by
% \begin{equation}
% \begin{aligned}
% E\colon y^2 &= x^3 - 5 g_4 x - \frac{7}{12} g_6,\\
% E'\colon y^2 &= x^3 - 5 (-\delta)^2 \iota^* g_4 x - \frac{7}{12} (-\delta)^3 \iota^* g_6.
% \end{aligned}
% \end{equation}

% \begin{proposition}
% There exists a normalized isogeny
% $$
% \phi\colon E\to E'
% $$
% whose kernel is the distinguished subgroup $C_\gamma\subset E$ defined by~$\gamma$.
% \end{proposition}

% Suppose that the hyperelliptic involution equals the Atkin--Lehner involution~$w_d$.  Then there is a unique meromorphic differential $\alpha_d$ on~$X_0(n)$ such that in terms of the formal coordinate~$q$, we can write
% $$
% \alpha_d=\frac{d}{12}\EisE_2^{(d)}(q)\frac{\dd q}{q}.
% $$
% This $\alpha_d$ satisfies
% $$
% w_d^* \alpha_d = -\alpha_d.
% $$

\section{Exceptional points}
\label{sec:exceptions}

\subsection{Notation}

In our tables, $\Q(\sqrt d)$ is the field of definition of the exceptional elliptic curve, and $w$ denotes $\sqrt d$.  The coordinates denote the quadratic points on the given model of $X_0(n)$ that are not in $\phi^{-1}\{0\}(\Q)$ in the notation of Section~\ref{sec:qp}.  The ``CM'' column indicates whether the corresponding elliptic curves have complex multiplication; an entry $-D$ means that they have complex multiplication by the imaginary quadratic order of discriminant~$-D$.

As we are interested in the moduli interpretation of all the quadratic points on $X_0(n)$, it makes sense to determine the isogeny diagrams corresponding to the exceptional points. In our diagrams, the vertices are the points~$P$ on $X_0(n)$ instead of the corresponding elliptic curves with an $n$-isogeny. We use this notation as the points~$P$ take considerably less space to write down. We do not write down the functions that for a given point on $X_0(n)$ construct an elliptic curve with an $n$-isogeny; these functions are implemented in Magma.

Let $P_1$, $P_2$ be points on $X_0(n)$ forming an isogeny class whose isogeny diagram is of the form
$$
P_1 \stackrel{\textstyle n}{\vcenter{\hbox to 2.5em{\hrulefill}}} P_2.
$$
In this case, we say that the isogeny diagram is \emph{simple} and denote it by $\Si(P_1,P_2,n)$.

Let $P_1$, $P_2$, $P_3$, $P_4$ be points on $X_0(n)$ forming an isogeny class whose isogeny diagram is of the form
\begin{center}
\begin{tikzpicture}
\matrix(m)[matrix of math nodes,
row sep=2.6em, column sep=2.8em,
text height=1.5ex, text depth=0.25ex]
{P_1&P_2\\
P_3&P_4\\};
\path
(m-1-1) edge node[auto] {$a$} (m-1-2)
edge node[auto] {$b$} (m-2-1)
(m-1-2) edge node[auto] {$b$} (m-2-2)
(m-2-1) edge node[auto] {$a$} (m-2-2);
\end{tikzpicture}
\end{center}
where the degrees $a$ and $b$ of the isogenies satisfy $ab=n$.  In this case, we say that the isogeny diagram is a \emph{square} and denote it by $\SQ(P_1,P_2,P_3,P_4,a,b)$.

For completeness, we give data on the modular curves $X_0(n)$ for all $n$ in our list, even though some of them do not have exceptional points.

\FloatBarrier

%\subsection{$n=22$}
\begin{table}
\caption{$X_0(22)$}
\begin{flushleft}
Model:
$y^2 + (-x^3)y = -x^4 + 5x^3 - 10x^2 + 12x - 8$

Genus: 2

Hyperelliptic involution: $w_{11}$

Group structure:
$J_0(22)(\Q)\simeq  \Z/5\Z \oplus \Z /5\Z$

Exceptional conjugacy classes of points:
\end{flushleft}
\setstretch{1.5}
\label{table:exceptional-22}
\begin{tabular}{cccc}
Name & $d$ & Coordinates & CM \\
\hline
$P_1$ & $-7$ & $\left(\frac 1 2(w - 1),\frac 1 2(w + 11)\right)$ & no\\
$P_2$ & $-7$ & $\left(\frac 1 2(w - 1),-w - 3\right)$ & no\\
$P_3$ & $-7$ & $\left(\frac 1 4(-w + 3),\frac 1 {16}(7w - 13)\right)$ & no\\
$P_4$ & $-7$ & $\left(\frac 1 4(-w + 3),\frac 1 4(-3w + 1)\right)$ & no\\
$P_5$  &$-7$ & $\left(\frac 1 2(-w + 1),\frac 1 2(w - 5)\right)$ & $-7$\\
$P_6$  &$-7$ & $\left(\frac 1 2(-w + 1),0\right)$ & $-7$\\
$P_7$ & $33$ & $\left(\frac 1 2(-w - 3),\frac 1 2(-3w - 13)\right)$ & no\\
$P_{8}$ & $33$ & $\left(\frac 1 2(-w - 3),-6w - 34 \right)$ & no\\
$P_{9}$ & $-47$ & $\left(\frac 1 4(w + 1),\frac 1 {16}(-7w + 1)\right)$ & no\\
$P_{10}$ & $-47$ & $\left(\frac 1 4(w + 1),\frac 1 4(-w - 9)\right)$ & no\\
$P_{11}$ & $-47$ & $\left(\frac 1 6(-w + 5),\frac 1 {27}(w - 41)\right)$ & no\\
$P_{12}$ & $-47$ & $\left(\frac 1 6(-w + 5),\frac 1 6(-w - 7)\right)$ & no\\
$P_{13}$ & $-2$ & $\left(w + 1,-1\right)$ & no\\
$P_{14}$ & $-2$ & $\left(w + 1,w - 4\right)$ & no\\
$P_{15}$ & $-2$ & $\left(\frac 2 3(-2w + 1),\frac 8 3(w - 2)\right)$ & no\\
$P_{16}$ & $-2$ & $\left(\frac 2 3(-2w + 1),\frac 8 {27}(w - 5)\right)$ & no\\
\end{tabular}

\medskip

Isogeny diagrams of non-CM points, up to conjugation:\\
$\SQ(P_1,P_2,P_3,P_4,11,2)$,
$\SQ(P_7,\lsup\sigma{P_7},P_{8},\lsup\sigma{P_{8}},11,2)$,
$\SQ(P_{9},P_{10},P_{11},P_{12},11,2)$,
$\SQ(P_{13},P_{14},P_{15},P_{16},11,2)$.
\begin{remark}
The points $P_7$ and $\lsup\sigma{P_7}$ are lifts of a $\Q$-point on $X_0(22)/w_2$, as are $P_8$ and $\lsup\sigma{P_8}$.
\end{remark}
\end{table}

%\subsection{$n=23$}
\begin{table}
\caption{$X_0(23)$}
\begin{flushleft}
Model:
$y^2 + (-x^3 - x - 1)y = -2x^5 - 3x^2 + 2x - 2$

Genus: 2

Hyperelliptic involution: $w_{23}$

Group structure:
$J_0(23)(\Q)\simeq  \Z/11\Z$

Exceptional conjugacy classes of points:
\end{flushleft}
\setstretch{1.5}
\label{table:exceptional-23}
\begin{tabular}{cccc}
Name & $d$ & Coordinates & CM \\
\hline
$P_1$ & $-5$ & $\left(\frac 1 3(2w - 2),\frac 1 9(8w + 70) \right)$ & no\\
$P_2$ & $-5$ & $\left(\frac 1 3(2w - 2),\frac 1 {27}(-22w - 89) \right)$ & no\\
$P_3$ & $-7$ & $\left(\frac 1 4(-w + 3),\frac 1 4(-w - 1) \right)$ & no\\
$P_4$ & $-7$ & $\left(\frac 1 4(-w + 3),\frac 1 {16}(-5w + 23) \right) $ & no\\
$P_{5}$ & $-11$ & $\left(\frac 1 6(-w + 1),\frac 1 {54}(-19w + 49) \right) $ & no\\
$P_{6}$ & $-11$ & $\left(\frac 1 6(-w + 1),\frac 1 9(2w + 1) \right) $ & no\\
$P_{7}$ & $-15$ & $\left(\frac 1 4(w + 1),\frac 1 {16}(-3w + 5) \right) $ & no\\
$P_{8}$ & $-15$ & $\left(\frac 1 4(w + 1),\frac 1 4(w + 1) \right) $ & no\\

\end{tabular}

\medskip

Isogeny diagrams of non-CM points, up to conjugation:\\
$\Si(P_1,P_2,23)$, %$\Si(\lsup\sigma{P_1},\lsup\sigma{P_2},23)$,
$\Si(P_3,P_4,23)$, %$\Si(\lsup\sigma{P_3},\lsup\sigma{P_4},23)$,
$\Si(P_5,P_{6},23)$, %$\Si(\lsup\sigma{P_7},\lsup\sigma{P_{8}},23)$,
$\Si(P_{7},P_{8},23)$. %$\Si(\lsup\sigma{P_{10}},\lsup\sigma{P_{11}},23)$.

\end{table}

%\subsection{$n=26$}
\begin{table}
\caption{$X_0(26)$}
\begin{flushleft}
Model:
$y^2 + (-x^3 - 1)y = -2x^5 + 2x^4 - 5x^3 +
    2x^2 - 2x$

Genus: 2

Hyperelliptic involution: $w_{26}$

Group structure:
$J_0(26)(\Q)\simeq  \Z/21\Z$

Exceptional conjugacy classes of points:
\end{flushleft}
\setstretch{1.5}
\label{table:exceptional-26}
\begin{tabular}{cccc}
Name & $d$ & Coordinates & CM \\
\hline
$P_1$ & $-3$ & $\left( \frac 1 2(w + 1),-1  \right)$ & $-3$\\
$P_2$ & $-3$ & $\left( \frac 1 2(w + 1),1  \right)$ & $-12$\\
$P_3$ & $-11$ & $\left( \frac 1 2(w - 1),7   \right)$ & no\\
$P_4$ & $-11$ & $\left( \frac 1 6(-w - 1),\frac 1 {27}(7w + 28)   \right)$ & no\\
$P_5$ & $-11$ & $\left( \frac 1 6(-w - 1),\frac 1 9(-2w + 1)   \right)$ & no\\
$P_6$ & $-11$ & $\left( \frac 1 2(w - 1),-w - 2   \right)$ & no\\
$P_7$ & $-23$ & $\left( \frac 1 6(-w + 1),\frac 1 {27}(w + 2)   \right)$ & no\\
$P_{8}$ & $-23$ & $\left( \frac 1 4(w + 1),\frac 1 {16}(-w + 3)   \right)$ & no\\
$P_{9}$ & $-23$ & $\left( \frac 1 4(w + 1),\frac 1 4(-w - 1)   \right)$ & no\\
$P_{10}$ & $-23$ & $\left( \frac 1 6(-w + 1),\frac 1 {18}(w + 11)   \right)$ & no\\
$P_{11}$ & $-1$ & $\left(w,1\right)$ & $-4$\\
$P_{12}$ & $-1$ & $\left(w,-w\right)$ & $-16$\\
\end{tabular}

\medskip

Isogeny diagrams of non-CM points, up to conjugation:\\
$\SQ(P_3,P_4,P_5,P_6,2,13)$, %$\SQ(\lsup\sigma{P_3},\lsup\sigma{P_4},\lsup\sigma{P_5},\lsup\sigma{P_6},2,13)$,
$\SQ(P_7,P_{8},P_{9},P_{10},2,13)$. %$\SQ(\lsup\sigma{P_7},\lsup\sigma{P_8},\lsup\sigma{P_9},\lsup\sigma{P_{10}},2,13)$,

\end{table}

%\subsection{$n=28$}
\begin{table}
\caption{$X_0(28)$}
\begin{flushleft}
Model:
$y^2 + (-2x^3 + 3x^2 - 3x)y = x^4 - 3x^3 + 4x^2 - 3x + 1$

Genus: 2

Hyperelliptic involution: $w_{7}$

Group structure:
$J_0(28)(\Q)\simeq  \Z/6\Z \oplus \Z/6\Z$

Exceptional conjugacy classes of points:
\end{flushleft}
\setstretch{1.5}
\label{table:exceptional-28}
\begin{tabular}{cccc}
\label{tab:28}
Name & $d$ & Coordinates & CM \\
\hline
$P_1$ & $-3$ & $\left( \frac 1 2(w + 1),0  \right)$ & $-12$\\
$P_2$ & $-3$ & $\left( \frac 1 2(w + 1),1 \right)$ & $-12$\\
$P_3$ & $-7$ & $\left( \frac 1 2(-w + 1),\frac 1 2(w + 1) \right)$ & $-7$\\
$P_4$ & $-7$ & $\left( \frac 1 4(w + 1),\frac 1 8(w + 5) \right)$ & $-7$\\
$P_5$ & $-7$ & $\left( \frac 1 4(-w + 3),\frac 1 8(-w + 3) \right)$ & $-28$\\
$P_{6}$ & $-23$ & $\left( \frac 1 4(w + 1),
\frac 1 8(-w + 19) \right)$ & no\\
$P_{7}$ & $-23$ & $\left( \frac 1 6(-w + 1),\frac 1 {54}(w + 29) \right)$ & no\\
$P_{8}$ & $-23$ & $\left( \frac 1 8(-w + 5),\frac 1 {64}(-3w + 7) \right)$ & no\\
$P_{9}$ & $-23$ & $\left( \frac 1 6(w + 5),\frac 1 {54}(-w + 25) \right)$ & no\\
$P_{10}$ & $-23$ & $\left( \frac 1 4(-w + 3),\frac 1 8(w - 11) \right)$ & no\\
$P_{11}$ & $-23$ & $\left( \frac 1 8(w + 3),\frac 1 {64}(3w + 57) \right)$ & no\\
$P_{12}$ & $-23$ & $\left( \frac 1 4(w + 1),\frac 1 8(-w + 3) \right)$ & no\\
$P_{13}$ & $-23$ & $\left( \frac 1 6(-w + 1),\frac 1 6(-w + 7) \right)$ & no\\
$P_{14}$ & $-23$ & $\left( \frac 1 8(-w + 5),\frac 1 {16}(-w + 13) \right)$ & no\\
$P_{15}$ & $-23$ & $\left( \frac 1 6(w + 5),\frac 1 6(w - 1) \right)$ & no\\
$P_{16}$ & $-23$ & $\left( \frac 1 8(w + 3),\frac 1 {16}(w + 3) \right)$ & no\\
$P_{17}$ & $-23$ & $\left( \frac 1 4(-w + 3),\frac 1 8(w + 5) \right)$ & no\\
%P_6 and P_7 - j_1
%P_8 and P_9 - j_2
%P_10 and P_11 - j_3'
%P_12 and P_13 - j_5'
%P_14 and P_15 - j_4'
%P_16 and P_17 - j_6
\end{tabular}

\medskip

Isogeny diagrams of non-CM points, up to conjugation:\\

\medskip
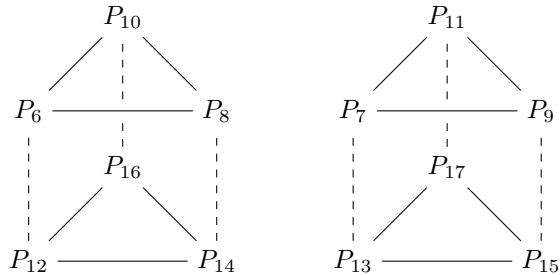
\begin{tabular}{cccc}
\begin{tikzpicture}[%
  back line/.style={densely dotted},
  cross line/.style={preaction={draw=white, -,line width=6pt}},
  node distance=2.5cm]
  \node (A) {$P_6$};
  \node [right of=A] (B) {$P_8$};
  \node [below of=A, node distance =2cm] (C) {$P_{12}$};
  \node [right of=C, node distance =2.5cm] (D) {$P_{14}$};
  \node (A1) [right of=A, above of=A, node distance=1.25cm] {$P_{10}$};
  \node [below of=A1, node distance =2cm] (C1) {$P_{16}$};
  \draw[dashed] (C1) -- (A1);
  \draw[dashed] (C) -- (A);
  \draw[cross line] (C) -- (C1) --(D)--(C);
  \draw[cross line](B) -- (A1) -- (A)  -- (B);
  \draw[dashed] (D) -- (B);
\end{tikzpicture}
& & &
\begin{tikzpicture}[%
  back line/.style={densely dotted},
  cross line/.style={preaction={draw=white, -,line width=6pt}},
  node distance=2.5cm]
  \node (A) {$P_7$};
  \node [right of=A] (B) {$P_9$};
  \node [below of=A, node distance =2cm] (C) {$P_{13}$};
  \node [right of=C, node distance =2.5cm] (D) {$P_{15}$};
  \node (A1) [right of=A, above of=A, node distance=1.25cm] {$P_{11}$};
  \node [below of=A1, node distance =2cm] (C1) {$P_{17}$};
  \draw[dashed] (C1) -- (A1);
  \draw[dashed] (C) -- (A);
  \draw[cross line] (C) -- (C1) --(D)--(C);
  \draw[cross line](B) -- (A1) -- (A)  -- (B);
  \draw[dashed] (D) -- (B);
\end{tikzpicture}
\end{tabular}\\
In these diagrams the dashed lines represent $7$-isogenies, while the full lines represent $4$-isogenies.
\end{table}

%\subsection{$n=29$}
\begin{table}
\caption{$X_0(29)$}
\begin{flushleft}
Model:
$y^2 + (-x^3 - 1)y = -x^5 - 3x^4 + 2x^2 + 2x - 2$

Genus: 2

Hyperelliptic involution: $w_{29}$

Group structure:
$J_0(29)(\Q)\simeq  \Z/7\Z$

Exceptional conjugacy classes of points:
\end{flushleft}
\setstretch{1.5}
\label{table:exceptional-29}
\begin{tabular}{cccc}
Name & $d$ & Coordinates & CM \\
\hline
$P_1$ & $-1$ & $\left( w - 1,2w + 4  \right)$ & no\\
$P_2$ & $-1$ & $\left( w - 1,- 1  \right)$ & no\\
$P_3$ & $-7$ & $\left( \frac 1 4(w + 1),\frac 1 {16}(-11w - 7)  \right)$ & no\\
$P_4$ & $-7$ & $\left( \frac 1 4(w + 1),\frac 1 8(5w + 9)  \right)$ & no\\
\end{tabular}

% NOTE: there are other CM points over Q(i), Q(-7) obtained by the FCM construction (rational x)
\medskip

Isogeny diagrams of non-CM points, up to conjugation:\\
$\Si(P_1,P_2,29)$, %$\Si(\lsup\sigma{P_1},\lsup\sigma{P_2},29)$,
$\Si(P_3,P_4,29)$. %$\Si(\lsup\sigma{P_3},\lsup\sigma{P_4},29)$,
\end{table}

%\subsection{$n=30$}
\begin{table}
\caption{$X_0(30)$}
\begin{flushleft}
Model:
$ y^2 + (-x^4 - x^3 - x^2)y = 3x^7 + 19x^6 +
    60x^5 + 110x^4 + 121x^3 + 79x^2 + 28x + 4$

Genus: 3

Hyperelliptic involution: $w_{15}$

Group structure:
$J_0(30)(\Q)\simeq  \Z/2\Z\oplus\Z/4\Z\oplus\Z/24\Z$

Exceptional conjugacy classes of points:
\end{flushleft}
\setstretch{1.5}
\label{table:exceptional-30}
\begin{tabular}{cccc}
Name & $d$ & Coordinates & CM \\
\hline
$P_1$ & $5$ & $\left( -w - 3,71w + 159  \right)$ & $-15$\\
$P_2$ & $5$ & $\left( \frac 1 2(-w - 3),4w + 9  \right)$ & $-60$\\
$P_3$ & $-7$ & $\left( \frac 1 2(-w - 3),w - 3  \right)$ & no\\
$P_4$ & $-7$ & $\left( \frac 1 4(w - 3),\frac 1 {32}(5w + 9)  \right)$ & no\\
$P_5$ & $-7$ & $\left( \frac 1 4(-w - 3),\frac 1 {16}(5w - 9)  \right)$ & no\\
$P_6$ & $-7$ & $\left( \frac 1 2(w - 3),\frac 1 2(w - 15)  \right)$ & no\\
\end{tabular}

\medskip

Isogeny diagrams of non-CM points, up to conjugation:\\

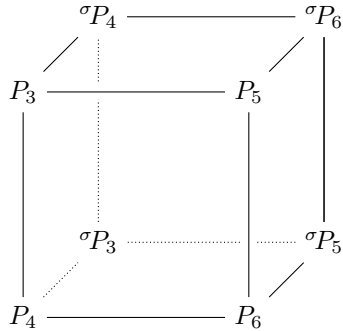
\begin{tikzpicture}[%
  back line/.style={densely dotted},
  cross line/.style={preaction={draw=white, -,line width=6pt}}, node distance =3cm]
  \node (A) {$P_3$};
  \node [right of=A] (B) {$P_5$};
  \node [below of=A] (C) {$P_4$};
  \node [right of=C] (D) {$P_6$};
  \node (A1) [right of=A, above of=A, node distance=1cm] {$\lsup\sigma{P_4}$};
  \node [right of=A1] (B1) {$\lsup\sigma{P_6}$};
  \node [below of=A1] (C1) {$\lsup\sigma{P_3}$};
  \node [right of=C1] (D1) {$\lsup\sigma{P_5}$};
  \draw[back line] (D1) -- (C1) -- (A1);
  \draw[back line] (C) -- (C1);
  \draw[cross line] (D1) -- (B1) -- (A1) -- (A)  -- (B) -- (D) -- (C) -- (A);
  \draw (D) -- (D1) -- (B1) -- (B);
\end{tikzpicture}\\
In this diagram, the horizontal lines are $5$-isogenies, the vertical lines 2-isogenies and the diagonal lines $3$-isogenies.

\begin{remark}
All the curves in the diagram are $\Q$-curves and are $6$-isogenous to their Galois conjugates and arise from rational points on the curve $X_0(30)/w_6$.
\end{remark}
\end{table}

\begin{table}
\caption{$X_0(31)$}
\begin{flushleft}
Model:
$y^2 + (-x^3 - x - 1)y = -2x^5 + x^4 + 4x^3 - 3x^2 - 4x - 1$

Genus: 2

Hyperelliptic involution: $w_{31}$

Group structure:
$J_0(31)(\Q)\simeq  \Z/5\Z$

Exceptional conjugacy classes of points:
\end{flushleft}
\setstretch{1.5}
\begin{tabular}{cccc}
Name & $d$ & Coordinates & CM \\
\hline
$P_1$ & $-3$ & $\left( \frac 1 2(w - 1),-2  \right)$ & no\\
$P_2$ & $-3$ & $\left( \frac 1 2(w - 1),\frac 1 2 (w + 7)  \right)$ & no\\

\end{tabular}

\medskip

Isogeny diagrams of non-CM points, up to conjugation:\\
$\Si(P_1,P_2, 31)$.

\end{table}

%\subsection{$n=33$}
\begin{table}
\caption{$X_0(33)$}
\begin{flushleft}
Model:
$y^2 + (-x^4 - x^2 - 1)y = 2x^6 - 2x^5 + 11x^4 - 10x^3 + 20x^2 - 11x + 8$

Genus: 3

Hyperelliptic involution: $w_{11}$

Group structure:
$J_0(33)(\Q)\simeq  \Z/10\Z\oplus\Z/10\Z$

Exceptional conjugacy classes of points:
\end{flushleft}
\setstretch{1.5}
\label{table:exceptional-33}
\begin{tabular}{cccc}
Name & $d$ & Coordinates & CM \\
\hline
$P_1$ & $-2$ & $\left( -\frac 1 2w,\frac 1 4(-4w - 5)  \right)$ & no\\
$P_2$ & $-2$ & $\left( w - 1,-5w - 5  \right)$ & no\\
$P_3$ & $-2$ & $\left( -\frac 1 2w,w + 2  \right)$ & no\\
$P_4$ & $-2$ & $\left( w - 1,7w - 2  \right)$ & no\\
$P_5$ & $-2$ & $\left( w,-w + 1 \right)$ & $-8$\\
$P_6$ & $-2$ & $\left( w,w + 2 \right)$ & $-8$\\
$P_7$ & $-7$ & $\left( \frac 1 4(-3w + 1),\frac 1 {32}(9w + 93) \right)$ & no\\
$P_{8}$ & $-7$ & $\left( \frac 1 2(w + 1),-1 \right)$ & no\\
$P_{9}$ & $-7$ & $\left( \frac 1 4(-3w + 1),\frac 1 4(9w + 33) \right)$ & no\\
$P_{10}$ & $-7$ & $\left( \frac 1 2(w + 1),-w + 1 \right)$ & no\\
$P_{11}$ & $-11$ & $\left( \frac 1 2(-w + 1),w + 1 \right)$ & $-11$\\

\end{tabular}

\medskip

Isogeny diagrams of non-CM points, up to conjugation:\\
$\SQ(P_1,P_2,P_3,P_4, 3, 11)$, %$\SQ(\lsup\sigma{P_1},\lsup\sigma{P_2},\lsup\sigma{P_3},\lsup\sigma{P_4}, 3, 11)$,
$\SQ(P_7,P_8,P_9,P_{10}, 3, 11)$. %$\SQ(\lsup\sigma{P_7},\lsup\sigma{P_8},\lsup\sigma{P_9},\lsup\sigma{P_{10}}, 3, 11)$.
\end{table}

%\subsection{$n=35$}
\begin{table}
\caption{$X_0(35)$}
\begin{flushleft}
Model:
$y^2 + (-x^4 - x^2 - 1)y = -x^7 - 2x^6 - x^5 - 3x^4 + x^3 - 2x^2 + x$

Genus: 3

Hyperelliptic involution: $w_{35}$

Group structure:
$J_0(35)(\Q) \simeq \Z/2\Z\oplus\Z/24\Z$

Exceptional conjugacy classes of points:
\end{flushleft}
\setstretch{1.5}
\label{table:exceptional-35}
\begin{tabular}{cccc}
Name & $d$ & Coordinates & CM \\
\hline
$P_1$ & $5$ & $\left( \frac 1 2(-w - 1),w + 3  \right)$ & $-35$\\
\end{tabular}

\medskip

Isogeny diagrams of non-CM points, up to conjugation:\\

\end{table}

%\subsection{$n=39$}
\begin{table}
\caption{$X_0(39)$}
\begin{flushleft}
Model:
$y^2 + (-x^4 - x^3 - x^2 - x - 1)y = -2x^7 + 2x^5 - 7x^4 + 2x^3 - 2x$

Genus: 3

Hyperelliptic involution: $w_{39}$

Group structure:
$J_0(39)(\Q) \simeq \Z/2\Z\oplus\Z/28\Z$

Exceptional conjugacy classes of points:
\end{flushleft}
\setstretch{1.5}
\label{table:exceptional-39}
\begin{tabular}{cccc}
Name & $d$ & Coordinates & CM \\
\hline
$P_1$ & $-3$ & $\left( \frac 1 2(-w + 1),\frac 1 2(-w + 1)  \right)$ & $-3$\\
$P_2$ & $-3$ & $\left( \frac 1 2(-w + 1),-1  \right)$ & $-27$\\
$P_3$ & $-7$ & $\left( \frac 1 4(-w + 3),\frac 1 32(-21w + 7)  \right)$ & no\\
$P_4$ & $-7$ & $\left( \frac 1 4(w + 3),\frac 1 8(3w + 1)  \right)$ & no\\

\end{tabular}

\medskip

Isogeny diagrams of non-CM points, up to conjugation:\\
$\SQ(P_3,\lsup\sigma{P_3}, P_4, \lsup\sigma{P_4}, 3, 13)$.

\begin{remark}
The points $P_3$ and $\lsup\sigma{P_3}$ are
lifts of a $\Q$-point on $X_0(39)/w_3$, as are $P_4$ and $\lsup\sigma{P_4}$.
\end{remark}

\end{table}

%\subsection{$n=40$}
\begin{table}

\caption{$X_0(40)$}
\begin{flushleft}
Model:
$y^2 + (-x^4 - 1)y = 2x^6 - x^4 + 2x^2$

Genus: 3

Hyperelliptic involution: induced by $\beta_{40} = \smallmat{-10}{1}{-120}{10}$

Group structure:
$J_0(40)(\Q) \simeq \Z/12\Z\oplus\Z/12\Z$

Exceptional conjugacy classes of points:
\end{flushleft}
\setstretch{1.5}
\label{table:exceptional-40}
\begin{tabular}{cccc}
\label{tab:40}
Name & $d$ & Coordinates & CM \\
\hline
$P_1$ & $-1$ & $\left( w,2w + 1   \right)$ & $-16$\\
$P_2$ & $-1$ & $\left( w,-2w + 1   \right)$ & $-16$\\

\end{tabular}

Isogeny diagrams of non-CM points, up to conjugation:\\

\end{table}

%\subsection{$n=41$}
\begin{table}
\caption{$X_0(41)$}
\begin{flushleft}
Model:
$y^2 + (-x^4 - x)y = -x^7 - 2x^6 + 2x^5 + 5x^4 + 2x^3 - 4x^2 - 5x - 2$

Genus: 3

Hyperelliptic involution: $w_{41}$

Group structure:
$J_0(41)(\Q) \simeq \Z/10\Z$

Exceptional conjugacy classes of points:
\end{flushleft}
\setstretch{1.5}
\label{table:exceptional-41}
\begin{tabular}{cccc}
Name & $d$ & Coordinates & CM \\
\hline
$P_1$ & $-1$ & $\left( \frac 1 2(-w - 1),\frac 1 4(-3w - 4)   \right)$ & no\\
$P_2$ & $-1$ & $\left( \frac 1 2(-w - 1),\frac 1 4(w + 1)   \right)$ & no\\

\end{tabular}

Isogeny diagrams of non-CM points, up to conjugation:\\
$\Si(P_1,P_2, 41)$.%, $\Si(\lsup\sigma{P_1},\lsup\sigma{P_2}, 41)$
\end{table}

\begin{table}
\caption{$X_0(46)$}
\begin{flushleft}
Model:
$$
\begin{aligned}
y^2 + (-x^6 - x^5 - x^3 - 1)y &= -x^{11} + x^{10} +
    x^9 - 7x^8 + 21x^7 - 29x^6 \\
& \qquad + 33x^5 - 16x^4 + 6x^3 + 3x^2 + 2x - 2
\end{aligned}
$$
Genus: 5

Hyperelliptic involution: $w_{23}$

Group structure:
$$J_0(46)(\Q) \simeq \Z/11\Z \oplus \Z/22\Z$$

Exceptional conjugacy classes of points:
\end{flushleft}
\setstretch{1.5}
\label{table:exceptional-46}
\begin{tabular}{cccc}
Name & $d$ & Coordinates & CM \\
\hline
$P_1$ & $-7$ & $\left(\frac 1 2(w + 1), w + 3\right)$ & $-7$\\
$P_2$ & $-7$ & $\left(\frac 1 2(w + 1), \frac 1 2(w + 11)\right)$ & $-7$\\
\end{tabular}
\medskip

Isogeny diagrams of non-CM points, up to conjugation:\\

\end{table}

\begin{table}
\caption{$X_0(47)$}
\begin{flushleft}
Model:
$y^2 + (-x^5 - x^4 - x^3 - x^2 - 1)y = -2x^9 + 2x^8 - 7x^7 + 4x^6 - 5x^5 - 4x^4 + 7x^3 - 10x^2 + 7x - 3$

Genus: 4

Hyperelliptic involution: $w_{47}$

Group structure:
$J_0(47)(\Q)\simeq \Z /23 \Z$

No exceptional points.
\end{flushleft}
\end{table}

%\subsection{$n=48$}
\begin{table}
\caption{$X_0(48)$}
\begin{flushleft}
Model:
$y^2 = x^8+14x^4+1$

Genus: 3

Hyperelliptic involution: induced by $\beta_{48} = \smallmat{-6}{1}{-48}{6}$

Group structure:
$J_0(48)(\Q) \simeq \Z/4\Z \oplus \Z/4\Z \oplus \Z/8\Z$

Exceptional conjugacy classes of points:
\end{flushleft}
\setstretch{1.5}
\label{table:exceptional-48}
\begin{tabular}{cccc}
Name & $d$ & Coordinates & CM \\
\hline
$P_1$ & $-1$ & $\left( w,4   \right)$ & $-16$\\
$P_2$ & $-1$ & $\left( w,-4   \right)$ & $-16$\\

\end{tabular}
\medskip

Isogeny diagrams of non-CM points, up to conjugation:\\

\end{table}

%\subsection{$n=50$}
\begin{table}
\caption{$X_0(50)$}
\begin{flushleft}
Model:
$y^2 + (-x^3 - 1)y = -x^5 - 3x^3 - x$

Genus: 2

Hyperelliptic involution: $w_{50}$

Group structure:
$J_0(50)(\Q)\simeq \Z/15\Z$

Exceptional conjugacy classes of points:
\end{flushleft}
\setstretch{1.5}
\label{table:exceptional-50}
\begin{tabular}{cccc}
Name & $d$ & Coordinates & CM \\
\hline
$P_1$ & $-1$ & $\left( w,-w   \right)$ & $-4$\\
$P_2$ & $-1$ & $\left( w,1   \right)$ & $-16$\\
$P_3$ & $-7$ & $\left( \frac 1 2(w - 1),3   \right)$ & no\\
$P_4$ & $-7$ & $\left( \frac 1 4(-w - 1),\frac 1 {16}(3w + 15)  \right)$ & no\\
$P_5$ & $-7$ & $\left( \frac 1 4(-w - 1),\frac 1 8(-w + 3)   \right)$ & no\\
$P_6$ & $-7$ & $\left( \frac 1 2(w - 1),\frac 1 2(-w + 1)  \right)$ & no\\

\end{tabular}

Isogeny diagrams of non-CM points, up to conjugation:\\
$\SQ(P_3,P_4,P_5,P_6,2,25)$
%$\SQ(\lsup\sigma{P_3},\lsup\sigma{P_4},\lsup\sigma{P_5},\lsup\sigma{P_6},2,25)$.
\end{table}

\begin{table}
\caption{$X_0(59)$}
\begin{flushleft}
Model:
$y^2 + (-x^6 - x^4 - x^2)y = -2x^{11} + 5x^{10} -
    7x^9 + 10x^7 - 16x^6 + 10x^5 - x^4 - 6x^3 + 5x^2 - x - 2$

Genus: 5

Hyperelliptic involution: $w_{59}$

Group structure:
$J_0(59)(\Q)\simeq \Z /29 \Z$

No exceptional points.
\end{flushleft}
\end{table}

\begin{table}
\caption{$X_0(71)$}
\begin{flushleft}
Model:
$$
\begin{aligned}
y^2 + (-x^7 - x^5 - x^4 - x^3 - 1)y
&= x^{13} - x^{12} - 10x^{11} - 20x^{10} - 7x^9 + 27x^8\\
&\qquad + 36x^7 - 31x^5 - 18x^4 + 7x^3 + 10x^2 + x - 3
\end{aligned}
$$
Genus: 6

Hyperelliptic involution: $w_{71}$

Group structure:
$J_0(71)(\Q)\simeq \Z/35\Z$

No exceptional points.
\end{flushleft}
\end{table}

\FloatBarrier

\end{document}